\documentclass[numbers=noenddot]{amsart} 

\usepackage{a4wide}
\usepackage{color}
\usepackage{amssymb, amsmath}
\usepackage{mathrsfs}

\usepackage[normalem]{ulem}

\usepackage{float}
\usepackage{graphicx}

\usepackage{tikz}
\usepackage{tikz-cd}
\usetikzlibrary{snakes}
\usetikzlibrary{intersections, calc}
\usepackage{epsfig}
\usepackage[all]{xy}
\usepackage{epstopdf}
\usepackage{framed}

\usepackage{subcaption}
\usepackage[toc,page]{appendix}
\newtheorem{theorem}{Theorem}[section]
\newtheorem{lemma}[theorem]{Lemma}
\newtheorem{proposition}[theorem]{Proposition}
\newtheorem{corollary}[theorem]{Corollary}

\theoremstyle{definition}
\newtheorem{definition}[theorem]{Definition}
\newtheorem{example}[theorem]{Example}

\newcounter{thmMaincounter}
\theoremstyle{theorem}
\newtheorem{theoremM}[thmMaincounter]{Theorem}

\newtheorem{probM}[thmMaincounter]{Problem}

\theoremstyle{remark}
\newtheorem{remark}[theorem]{Remark}

\numberwithin{equation}{section}

\newcommand{\C}{{\mathbb{C}}}

\newcommand{\Z}{{\mathbb{Z}}}
\newcommand{\Q}{{\mathbb{Q}}}
\newcommand{\R}{{\mathbb{R}}}

\DeclareMathOperator{\Aut}{Aut}

\DeclareMathOperator{\chr}{char}

\newcommand{\eps}{\varepsilon}
\newcommand{\lb}{\lbrace}
\newcommand{\rb}{\rbrace}
\newcommand{\la}{\langle}
\newcommand{\ra}{\rangle}

\begin{document}

\title[Independence and symmetric spaces]{Independence of homogeneous GKM manifolds and symmetric spaces}

\author{Shintar\^o Kuroki}
\address[S.\ Kuroki]{Department of Applied Mathematics Faculty of Science, Okayama University of Science, 1-1 Ridai-Cho Kita-Ku Okayama-shi Okayama 700-0005, Okayama, Japan}
\email{kuroki@ous.ac.jp}

\author{Grigory Solomadin (Corresponding author)}
\address[G.\,Solomadin]{Philipps-Universit\"at Marburg, Germany}
\email{grigory.solomadin@gmail.com}

\begin{abstract}
Let $G/H$ be a simply connected homogeneous space of maximal rank. 
Then the maximal torus $T$-action on $G/H$ equips it with a structure of a Goresky-Kottwitz-MacPherson (GKM) manifold.
We call the $T$-action $j$-independent if any $i\leq j$ pairwise distinct isotropy weights at a fixed point are linearly independent.
Using weighted graphs, we show that the maximal independence of $G/H$ is $2$, $3$ or $n=\dim T$, and that the cases of $3$ or $n=\dim T$ correspond to some symmetric spaces of rank $>2$ by using the Lie bracket structure at the tangent space at identity coset.
We classify all $4$-independent homogeneous spaces by studying the $3$- and $4$-cycles in the matroid of the respective root subsystem complement using Borel-de Siebenthal and Cartan classification results.
In the classical $G$ cases our case study of maximal independence relies on signed weighted graphs, following the classical methods of Zaslavsky and Harari.
In the exceptional simply laced $G$ cases we either present a $3$-cycle, or prove that there is none by reducing them only to those having a form $x+y=z$ for positive roots $x,y,z$ and find an affine hyperplane containing the respective root complement.
As a corollary, 
using the results of Ayzenberg and Masuda, the lower-degree reduced rational cohomology groups of the orbit space $T\backslash G/H$ vanish.
\end{abstract}


\maketitle

\section{Introduction}
\label{sect:1}

\subsection{Motivation and Main theorem}
\label{sect:1.1}

Let $M$ be a $2m$-dimensional, smooth $T$-manifold with a nonempty set of isolated fixed points $M^{T}$, where $T\cong (S^{1})^{n}$ and we assume that the $T$-action is {\it almost effective}, i.e., the kernel of the $T$-action is finite.
There is the following irreducible decomposition 
\begin{align}
\label{cpx_on_p}
T_{p}M\cong \bigoplus_{i=1}^{m}V_{\alpha_{p,i}},
\end{align}
of the isotropy representation at any fixed point $p\in M^{T}$, where $V_{\alpha_{p,i}}$ is the (complex) irreducible representation space of $T$, and $\alpha_{p,i}\in \mathfrak{t}_{\mathbb{Z}}^{*}/\pm 1$ for $i=1,\ldots, m$ is determined up to a sign.
Here, $\mathfrak{t}_{\mathbb{Z}}^{*}\cong \Z^{n}$ denotes the character lattice of 
the dual of the Lie algebra $\mathfrak{t}$ of $T$.
Set $S_{p}:=\{\alpha_{p,i}\ |\ i=1,\ldots, m\}\subset \mathbb{Z}^{n}$ and define the following numbers:
\begin{align}
\label{independency}
k(p):=\max\{q\ |\ \text{every $q$ elements in $S_{p}$ are linearly independent}\},\ 
k(M):=\min\limits_{p\in M^{T}} k(p).
\end{align}
If $2\le k(M)$, then this manifold is called a \textit{GKM manifold}~\cite{GKM},~\cite{GZ}. 
A GKM manifold is called a \textit{$j$-independent} or \textit{GKM$_{j}$) manifold} if $k(M)\geq j$ holds (see~\cite{AMS, GS}). In~\cite{AM}, this is called a {\it $j$-generality}). 
If $k(M)=m$ then $n=m$; in such case this manifold is called a {\it torus manifold}~\cite{MP},~\cite{HaMa}.  
Mikiya Masuda has asked the following question to the first named author: 
\begin{probM}
\label{Masuda-problem}
Let $M$ be an equivariantly formal (i.e., $H^{odd}(M)=0$) $2m$-dimensional GKM$_4$ manifold with a $T$-action.
Is $M$ homeomorphic to a torus manifold, and is the given $T$-action obtained as the restriction of a half-dimensional torus $T^{m}$-action to a subtorus $T$?
\end{probM}

The main result of the present paper is given as follows (see Theorem~\ref{main}), where by a {\it maximal rank compact symmetric space} we mean a compact symmetric space $G/H$ that is a homogeneous space of maximal rank (also referred to as a {\it class 1 symmetric space of inner type}).
In what follows, we refer to homogeneous manifolds by the types of the respective Lie algebras.
In particular, for all $n\in\Z_{\geq 0}$ we consider simple compact Lie groups of the classical types $A_{n}$, $B_{n}$, $C_{n}$, $D_{n}$, and of the exceptional types
$G_{2}$, $F_{4}$, $E_{6}$, $E_{7}$, $E_{8}$.
Furthermore, we consider their Cartesian products, their coset spaces and covers.
Taking covers does not change $k(M)$; therefore, Theorem~\ref{thmM:mp} is valid for any choice of such spaces.
Unless specifically claimed otherwise, in the above classical cases we take $SU(n+1)$,  $SO(2n+1)$, $Sp(n)$ and $SO(2n)$, respectively.

\begin{theoremM}\label{thmM:mp}
Let $G/H$ be any simply connected, homogeneous space such that $T\subset H\subset G$ is a maximal torus, ${\rm rank}\ T=n$, and $G$ is a compact, connected, simple Lie group.
Then one has 
\[
k(G/H)=2,3\ \text{or}\ n.
\]
Furthermore, 
$k(G/H) = 3$ or $n$ if and only if $G/H$ is a maximal rank compact symmetric space with rank $>2$ other than $B_{n}/(D_{i}\times B_{n-i})$ for all $1\leq i<n$, $C_{n}/(A_{n-1}\times T^{1})$, $F_{4}/(C_{3}\times A_{1})$.
\end{theoremM}
Note that by the classification of the homogeneous torus manifolds in~\cite{Ku09}, $k(G/H) =n\geq 4$ if and only if $G/H$ is equivariantly diffeomorphic to $A_{n}/(A_{n-1}\times T^{1})(\cong \mathbb{C}P^{n})$ or $B_{n}/D_{n}(\cong S^{2n})$.
Moreover, the independence of $M_{1}\times \cdots \times M_{r}$ coincides with 
the minimal independence of $M_{i}$'s (if it is not a torus manifold, see Proposition~\ref{split_computation}), and the independence of any connected homogeneous space $G/H$ is equal to that of its universal covering space $G_{1}/H_{1}\times \cdots \times G_{r}/H_{r}$.
Therefore, together with these facts, Theorem~\ref{thmM:mp} answers Problem~\ref{Masuda-problem} positively for the class of homogeneous GKM manifolds.

As an application of Theorem~\ref{thmM:mp}, together with the acyclicity results of~\cite{AMS},~\cite{AM}, we obtain vanishing (in degrees $\leq 4$) for rational reduced cohomology of the orbit spaces for homogeneous GKM manifolds, see Corollary~\ref{cor:acyc}.
Our results are consistent with those in~\cite{S19},~\cite{S21},~\cite{BT}, while providing new insights, particularly for several symmetric spaces.

\subsection{Sketch of the proof}
\label{sect:1.2}

The proof of Theorem~\ref{thmM:mp} follows by computing $k(G/H)$ for every homogeneous GKM manifold $G/H$.
Here, $k(G/H)=k(gH)$ for every fixed point $gH\in (G/H)^{T}$ (see~\cite{GHZ}).
We prove that $k(G/H)\geq 3$ implies that $G/H$ is a maximal rank compact symmetric space with rank $>2$ using the well known characterization of symmetric spaces in terms of the Lie algebra $T_{eH} G/H$ (see Lemma~\ref{auxiliary lemma}).
For maximal rank compact symmetric space $G/H$ such that $G$ is a classical Lie group (i.e., type $A$--$D$), the problem of finding maximal independence for $G/H$ reduces to the study of the matroid formed by the complement $\Delta^{+}_{G,H}:=\Delta^{+}_{G}\setminus \Delta^{+}_{H}$ of positive root systems.
The analysis of circuits for such a matroid is possible using {\it weighted signed graphs}, which were introduced in~\cite{Ha} and studied in~\cite{Za} (also see~\cite{FGLP} and~\cite{FGGP}).
We apply this technique together with Borel-de Siebenthal theory~\cite{BoSi} and Cartan's classification of symmetric spaces~\cite{C27} to compute $k(G/H)$ for all pairs $(G,H)$ except for $G$ being of type $F_{4},E_{6},E_{7},E_{8}$.
(We remark that signed graphs under the name of Levi graphs or crystallographs were also used in~\cite{Re} to study root subsystems.
The difference of our application for signed graphs is that we study \textit{complements} of root subsystems.)
In the remaining cases, i.e. for $G$ being of type $F_{4},E_{6},E_{7},E_{8}$, the proof is a simple case study.

\subsection{Structure of the paper}
\label{sect:1.3}

We briefly outline the structure of the present paper.
In Section~\ref{sec:prep}, we recall the isotropy representation for homogeneous GKM manifolds, Borel-de Siebenthal theory, Cartan's classification of symmetric spaces and homology vanishing of the orbit space for independent GKM manifolds. 
One of the main contributions of this paper is to clarify the relationship between symmetric spaces and GKM manifolds from the viewpoint of independence.
The value of $k(G/H)$ for each homogeneous space $G/H$ is listed in Theorem~\ref{main}. 
This result implies $\widetilde{H}^{i}(T\backslash G/H;\Q)=0$ for all $i\leq k(G/H)+1$ (see Corollary~\ref{cor:acyc} for symmetric GKM$_{3}$ cases).
This homology vanishing is compared with several previously known results in Examples~\ref{ex:origra} and~\ref{ex:comgra}.
The remainder of the paper is the proof of Theorem~\ref{main}.
In Section~\ref{sect:4}, linearly dependent triples are described in several cases in terms of weighted signed graphs (see Lemma~\ref{lm:23dep}).
In Sections~\ref{sect:5}--~\ref{sect:7}, we analyse linear dependencies for root subsystems of classical types, using Borel-de Siebenthal theory and weighted signed graphs.
In Sections~\ref{sect:8}--~\ref{sect:9},
the remaining cases (when $G$ is of type $F_{4},E_{6},E_{7},E_{8}$) are studied. 

\

\noindent {\bf Acknowledgements:} The authors gratefully acknowledge funding of the Deutsche Forschungsgemeinschaft
(DFG, German Research Foundation): Project number 561158824 (the second named author's Walter Benjamin Fellowship). 
We thank M.~Masuda for many valuable discussions on toric manifolds without whom we would not be able to conduct this project, and to O.~Goertsches for the interest to our paper.
We also thank M.~Nakagawa for giving us a useful information on symmetric spaces.
The first author was partially supported by JSPS KAKENHI Grant Number 21K03262.
He would like to thank the Philipps University of Marburg for providing him with an excellent environment for research during his stay in August 2025.
The second named author is grateful to Okayama University of Science  and Osaka Metropolitan University for a very comfortable work atmosphere, where the initial part of the present work was performed during his research visits in 2023. 
The authors thank the anonymous referees for valuable comments and meticulous reading of this work.

\section{Preparations}\label{sec:prep}
We first prepare the basic notions and facts to prove our main theorem.

\subsection{$k$-independence}
\label{sect:2.1}
Let $S:=\{\alpha_{1},\ldots, \alpha_{m}\}\subset \mathbb{Z}^{n} \subset \mathbb{R}^{n}$ be a set of vectors. 
We call $S$ a {\it $k$-independent vectors} for some $2\le k< n$ if every $k$ vectors $\alpha_{i_{1}},\ldots, \alpha_{i_{k}}\in S$ are linearly independent but there exists $k+1$ vectors $\alpha_{j_{1}},\ldots, \alpha_{j_{k+1}}\in S$ which are not linearly independent.

\begin{example}
\label{ex_matroid}
Let $S:=\{e_{1}, e_{2}, -e_{1}-e_{2}, e_{1}+e_{2}+e_{3}\}\subset \mathbb{R}^{3}$ be the set of $4$ vectors in the $3$-dimensional real vector space $\mathbb{R}^{3}$, where $e_{i}$, $i=1,2,3$, is the standard basis of $\mathbb{R}^{3}$.
Then, one can easily check that $S$ has the property $U_{4}^{2}$, i.e., $S$ is $2$-independent.
Note that the subset $\{e_{1},e_{2},e_{1}+e_{2}+e_{3}\}$ is linearly independent but $\{e_{1},e_{2},-e_{1}-e_{2}\}$ is not linearly independent.
\end{example}

\subsection{Some remarks on $k$-independence of tangential representations}
\label{sect:2.1}

Let $M$ be a $2m$-dimensional manifold with an $n$-dimensional torus $T$-action.
We assume that there is a non-empty fixed point set $M^{T}$, and the tangential representation at $p\in M^{T}$ is a direct sum of pairwise linearly independent characters of $T$.
Namely, for every $p\in M^{T}$, the irreducible decomposition (with respect to some complex structure) \eqref{cpx_on_p}
satisfies that $\alpha_{i}, \alpha_{j}$ for every $1\le i<j\le m$ are linearly independent,
where $V(\alpha_{i})$ is a complex $1$-dimensional irreducible representation space of the $T$-action and $\alpha_{i}\colon T\to S^{1}\in {\rm Hom}(T,S^{1}) \simeq \mathfrak{t}_{\mathbb{Z}}^{*}$. 
In this paper, we call $M^{2m}$  a {\it GKM manifold} (following~\cite{GZ}) if there is a $T^{n}$-action on $M^{2m}$ with the finite kernel such that its one-skeleton (the orbit space of the one- and zero-dimensional orbits) has the structure of a graph.

\begin{example}
\label{ex_depends_k(p)}
Recall that by~\eqref{independency}, $k(p)$ denotes the independence of the isotropy weights at the fixed point $p\in M^{T}$.
Now we may construct a GKM manifold $M$ 
such that there are two fixed points $p,\ q\in M^{T}$ with different properties $U_{m}^{k(p)}$ and $U_{m}^{k(q)}$.
Let $\mathbb{C}P^{1}$ be the complex projective space with the natural $T^{1}$-action.
We assume that $T^{1}\subset T^{3}$ is the first coordinate. 
Define the following symbols:
\begin{itemize}
\item $p_{i}\colon T^{3}\to S^{1}$ for $i=0,1,2,3$ is the projection onto the $i$-th coordinate, where $p_{0}$ is the trivial homomorphism;
\item $\epsilon_{i}=\mathbb{C}P^{1}\times \mathbb{C}$ for $i=0,1,2,3$ is the $T^{3}$-equivariant trivial line bundle over $\mathbb{C}P^{1}$ with the $T^{3}$-representation on the fiber $\mathbb{C}$ by the projection $p_{i}\colon T^{3}\to S^{1}$;
\item $\gamma$ is the tautological line bundle over $\mathbb{C}P^{1}$.
\end{itemize}
Let $M^{8}:=\mathbb{P}((\gamma\otimes \epsilon_{2}\otimes\epsilon_{3})\oplus \epsilon_{2}\oplus \epsilon_{3}\oplus \epsilon_{0})$ be the projectivization of the rank $4$ equivariant complex vector bundle over $\mathbb{C}P^{1}$, i.e., $M^{8}$ is the $\mathbb{C}P^{3}$-bundle over $\mathbb{C}P^{1}$ with $T^{3}$-action (see e.g.~\cite{KS}).
Then, 
by computing the tangential weights around $8$ fixed points, one can check that $M^{8}$ is a GKM manifold, i.e., for every fixed point $p$, the tangential weights have the property $U_{4}^{k(p)}$ with $k(p)\ge 2$.
Moreover, the tangential weights on the fixed points $p=([1:0],[0:0:0:1])$ 
and $q=([0:1],[0:0:0:1])$ are as follows:
\begin{align*}
& T_{p}M=V(\alpha)\oplus V(\beta+\gamma)\oplus V(\beta)\oplus V(\gamma); \\
& T_{q}M=V(-\alpha)\oplus V(\alpha+\beta+\gamma)\oplus V(\beta)\oplus V(\gamma).
\end{align*}
In this case, $k(p)=2$ and $k(q)=3$.
\end{example}

\begin{remark}
\label{sign_ambiguities}
Let $M$ be a $2m$-dimensional GKM manifold with the $n$-dimensional torus $T$-action.
If there is a $T$-invariant almost complex structure $J$ on $M$, then the identification \eqref{cpx_on_p} can be induced from $J$, see~\cite{GZ, GHZ}.
Also see~\cite[Section 12 and Proposition 13.5]{BoHi} about $G$-invariant (almost) complex structure of $G/H$.
For the case of $m=n$, called a {\it torus manifold} (e.g.~$M=S^{2n}\simeq SO(2n+1)/SO(2n)$ with the standard $T^{n}$-action for $n\ge 2$),
the identification \eqref{cpx_on_p} is determined by choosing the omni-orientation of the characteristic submanifolds in the torus manifold, see~\cite{MMP, Ku16}.
For the other cases (e.g.~$M=\mathbb{H}P^{n}\simeq Sp(n+1)/(Sp(n)\times Sp(1))$ for $n\ge 2$), to avoid the sign ambiguity of the representations, 
we may choose some identification \eqref{cpx_on_p} for each fixed point.
\end{remark}

\subsection{The torus action on the homogeneous space $G/H$}
\label{sect:2.3}
Let $G$ be a compact, connected, semi-simple Lie group, i.e., $\pi_{1}(G)$ is finite, and 
$T$ be its maximal torus, and $H$ be a closed, connected subgroup of $G$ such that $T\subseteq H\subset G$, i.e., a {\it maximal rank subgroup} of $G$.
It is known in this case that 
the homogeneous space $G/H$ is a $2m$-dimensional, simply connected manifold satisfying $H^{odd}(G/H;\mathbb{Q})=0$, see e.g.~\cite{MiTo, Bu} and~\cite[Theorem 2.4]{GHZ}.
Below we recall the description of the tangential weights of the $T$-action on $G/H$ by following these references.

We first recall some basic facts from the theory of Lie groups.
We denote the Lie algebras of $T\subseteq H\subset G$ by $\mathfrak{t}\subseteq \mathfrak{h}\subset \mathfrak{g}$.
Let $\exp\colon \mathfrak{g} \to G$ be the exponential map, and $\Lambda_{G}:=\exp^{-1}(e)\subset \mathfrak{t}$ be the {\it integer lattice}, where $e\in T\subset G$ is the identity element.
Then, we have an identification $T\simeq \mathfrak{t}/\Lambda_{G}$ and ${\rm Hom}(T,S^{1})\simeq \Lambda_{G}^{*}:= {\rm Hom}(\Lambda_{G},\mathbb{Z})\simeq \mathbb{Z}^{n}$.
The $\mathbb{Z}$-module $\Lambda^{*}_{G}\subset \mathfrak{t}^{*}$ is called a {\it weight lattice} of $G$, where $\mathfrak{t}^{*}={\rm Hom}(\mathfrak{t},\mathbb{R})$ is the dual of $\mathfrak{t}$.
Note that if $G$ is simply connected, then $\Lambda^{*}_{G}$ is spanned by the fundamental weights, which are dual to the simple coroots (see~\cite[Chapter 5 Theorem 6.36]{MiTo}).
If $G$ is not simply connected, then by taking the universal covering $p\colon\widetilde{G}\to G$ and the dual $dp^{*}$ of its differential $dp$ on the identities of maximal tori, 
we may regard $\Lambda_{G}^{*}\subset \Lambda_{\widetilde{G}}^{*}$ as finite index $>1$ embedding of free abelian subgroups (see~\cite[Chapter 5 Theorem 4.9]{MiTo} or~\cite[Section 19]{Bu}). 
(Here we identify the map with the induced map for the de Rham cohomology groups $H^{1}(T;\R)\to H^{1}(\widetilde{T};\R)$, restricted to the respective lattices.)
Let $\Delta_{G}\subset \Lambda_{G}^{*}\subset \mathfrak{t}^{*}$ be the root system of $G$ with respect to the maximal torus $T$.
Because the linear relations among the root systems do not change by taking the finite covering of $G$, that is, one has $k(\widetilde{G}/H)=k(G/H)$, it is enough to consider the case when $G$ is simply connected for the purpose of this paper.
Therefore, we can assume $\pi_{1}(G)=0$ in this paper without loss of generality.

Let $G$ be a simply connected, compact Lie group.
The set of the fixed points $(G/H)^{T}$ can be identified with the finite set of the quotient $W_{G}/W_{H}$ of the Weyl groups for the maximal torus $T$, where $W_{G}:=N_{G}(T)/T$ and $W_{H}:=N_{H}(T)/T$.
Let $p_0=eH\in (G/H)^{T}\simeq W_{G}/W_{H}$ be the coset of the identity element $e\in G$.
In other words, $W_{G}$ acts on $(G/H)^{T}$ transitively, where the isotropy subgroup of $p_{0}$ is $W_{H}$. 
Then, the tangential representation at $wp_{0}$ is 
\begin{align*}
T_{wp_0}G/H=\bigoplus_{[\beta]\in \Delta_{G,H}/\{\pm 1\}} V_{[w\beta]},
\end{align*}
where $\Delta_{G}^{+}\setminus \Delta_{H}^{+}\simeq \Delta_{G,H}/\{\pm 1\}$ for positive roots $\Delta_{H}^{+}\subset \Delta_{G}^{+}$.
Note that $\Delta_{p_{0}}$ and $\Delta_{wp_{0}}$ for every $w\in W_{G}/W_{H}$
are isomorphic to $\Delta_{G,H}/\{\pm 1\}$ by being in the orbit of the  Weyl group transitive action. 
This implies the following proposition:

\begin{proposition}
\label{k-independence}
Let $G$ be a compact, connected semi-simple Lie group, $H$ be its maximal rank subgroup and $T\subseteq H\subset G$ be a maximal torus. 
Then $k(G/H)=k(gH)$ for any $gH\in (G/H)^{T}$.
\end{proposition}
This proposition shows that if we compute the $k(G/H)$-independence on $\Delta_{G}^{+}\setminus \Delta_{H}^{+}\simeq \Delta_{G,H}/\{\pm 1\}$, then we have the $k(G/H)$-independence of all tangent spaces on the fixed points of $G/H$
(Cf.~Example~\ref{ex_depends_k(p)}).
In this paper, we determine the $k$-independence for all maximal rank homogeneous spaces $G/H$.

\subsection{The Borel-de Siebenthal theory}
\label{sect:2.4}

Let $G$ be a compact, connected semi-simple Lie group, $H$ be its maximal rank subgroup.
The classification of such pair $(G, H)$ is known by the work of Borel-de Siebenthal~\cite{BoSi}. 
In this subsection, we briefly recall the Borel-de Siebenthal theory.

Borel-de Siebenthal theory~\cite{BoSi} describes maximal subgroups of maximal rank in a simple Lie group $G$ as normalizers of certain elements or circle subgroups in $G$.
The classification of the respective root subsystems is given in terms of (ordinary or extended) Dynkin diagram of $G$ as follows.
(Below we follow notation from~\cite{MiTo} and tables from~\cite{Y25}.)
Recall that $\alpha_{1},\dots,\alpha_{n}\in\mathfrak{g}^{*}$ are simple roots of $G$, and the dominant root $\widetilde{\alpha}$ can be written as
\begin{align}
\label{dominant root}
\widetilde{\alpha}=\sum_{i=1}^{n} m_{i}\alpha_{i}
\end{align}
for some non-negative integers $m_{1},\ldots, m_{n}(\ge 0)$.
We set the additional notation by denoting by $X_{i}$ the vertex opposed to the facet with normal vector $\alpha_{i}$ of the simplex 
\[
S(G):=\lb X\in \mathfrak{g}\mid \alpha_{i}(X)\geq 0,\ i=1,\dots,n,\ \widetilde{\alpha}(X)\leq 1\rb\subset \mathfrak{g}.
\]
Also let $x_{i}:=\exp(X_{i})\in G$, and let $N(x)$ be the normalizer of $x\in G$.

\begin{theorem}[{\cite[Thm 7.16]{MiTo}}]
\label{thm:BdS}
The Dynkin diagram of a maximal subgroup of maximal rank for a simple Lie group $G$ is obtained from that of $G$ as follows:
\begin{enumerate}
\item If $m_{i}$ in \eqref{dominant root} is prime, then the Dynkin diagram of $G_{i}=N(x_{i})$ is obtained by deleting $\alpha_{i}$ from the extended Dynkin diagram, which is the Dynkin diagram of $G$ with $-\widetilde{\alpha}$ added;
\item If $m_{i}=1$ in \eqref{dominant root}, then $N(\exp X_{i}/2)$ is locally a direct product of $S^{1}$ and a semisimple subgroup $H$ of corank $1$. 
The Dynkin diagram of $H$ is obtained by deleting $\alpha_{i}$ from the Dynkin diagram of $G$.
\end{enumerate}
\end{theorem}

\begin{remark}
By using the Borel-de Siebenthal theory,
on the level of Dynkin diagrams, a maximal rank subgroup $H(\subset G)$ is determined by a sequence of certain added or removed vertices.
The isomorphism class of a subgroup does not depend on the choice of such a sequence, 
though the embedding depends on the order of such operations.
However, it turns out that the choice of embedding does not affect independence.
\end{remark}

\subsection{Symmetric spaces}
\label{sect:2.5}

In order to state our main theorem, we also need the notion of symmetric spaces. Here, we quickly recall it.

\begin{definition}[{\cite[Chapter III, Section 6]{MiTo}}]
A homogeneous space $G/H$ admitting an involution $\sigma\in\Aut G$ is called a \textit{symmetric space}, namely, $G^{\sigma}_{e}\subseteq H\subset G^{\sigma}=\lb g\in G\mid \sigma g=g\rb$ holds, where $G^{\sigma}_{e}$ is the identity component of $G^{\sigma}$.
\end{definition}

There is a well-known criterion for $G/H$ to be a symmetric space in terms of Lie algebras, which we recall next.
Let $\mathfrak{g}$ be a
Lie algebra of a simply connected, compact, simple Lie group $G$.
It has a decomposition into the Cartan subalgebra $C(\mathfrak{g})$ of $\mathfrak{g}$ and two-dimensional root planes $\mathfrak{g}_{\alpha}$ of $C(\mathfrak{g})$ as follows:
\begin{equation}\label{eq:liedec}
\mathfrak{g}=C(\mathfrak{g})\oplus\bigoplus\limits_{\alpha\in\Delta^{+}_{G}} \mathfrak{g}_{\alpha}.
\end{equation}
Let $\mathfrak{h}\subset\mathfrak{g}$ be a maximal-rank Lie subalgebra containing the chosen Cartan subalgebra $C(\mathfrak{g})$. 
Then, $\mathfrak{h}$ is a direct subsummand of \eqref{eq:liedec} determined by the root subsystem $\Delta_H\subseteq\Delta_G$.
Since $\mathfrak{h}$ is a Lie subalgebra, one has
\[
[\mathfrak{h},\mathfrak{h}]\subset \mathfrak{h}.
\]
Since $\mathfrak{g}$ is semi-simple (and therefore reductive Lie algebra), the Killing form on $\mathfrak{g}$ is nondegenerate, and the respective orthogonal complement to $\mathfrak{h}$ in $\mathfrak{g}$, say $\mathfrak{m}$, is $\mathfrak{h}$-invariant.
Therefore, one has
\[
[\mathfrak{h},\mathfrak{m}]\subset \mathfrak{m}.
\]

Recall the following proposition.
\begin{proposition}[{\cite[\S B.1]{Ba}}]\label{pr:symm}
Let $\mathfrak{h}$ be any maximal-rank Lie subalgebra containing the chosen Cartan subalgebra of any semisimple Lie algebra $\mathfrak{g}$.
Then the following are equivalent:
\begin{enumerate}
\item $ [\mathfrak{m},\mathfrak{m}]\subset \mathfrak{h}$, where $\mathfrak{m}$ is the orthogonal complement to $\mathfrak{h}$ with respect to the Killing form;
\item The Lie algebra $\mathfrak{g}$ has an involution $\sigma$ to which $\mathfrak{h}$ is the eigenspace with value $1$ and $\mathfrak{m}$ is the eigenspace with value $-1$ in the direct sum $\mathfrak{h}\oplus \mathfrak{m}$;
\item The homogeneous space $G/H$ is a symmetric space, where $G$ is a Lie group and $H$ is its closed subgroup such that the Lie algebras of $G$ and $H$ are $\mathfrak{g}$ and $\mathfrak{h}$, respectively.
\end{enumerate}
\end{proposition}

The classification of symmetric spaces was obtained by E.~Cartan~\cite{C27}.
In this paper, we restrict our attention to the symmetric spaces that are compact and have nonzero Euler characteristic with ${\rm rank}(G)={\rm rank}(H)>2$, i.e., maximal rank compact symmetric spaces with rank $>2$. 
Thus, by Cartan's classification, we may only consider the following symmetric spaces
(see e.g.~\cite[Ch.X, \S6, Table V]{He} for classical types, and~\cite[Ch.V, \S7, Remark 7.18--7.19]{MiTo} and~\cite{Y25} for exceptional types).

\begin{theorem}
\label{thm:cart}
Let $G/H$ be a maximal rank compact symmetric space with rank $>2$.
Then, the pair $(G,H)$ corresponds to one of the following entries.
\begin{center}
\begin{tabular}{|l|l|l|l|l|} \hline
   label & pair & $G$ & $H$ &  \\ \hline
   {\rm AIII} & $(A_{n},A_{i}\times A_{n-1-i}\times T^{1})$ & $SU(n+1)$ & $S(U(i+1)\times U(n-i))$ & $0\le i\le n-1$ \\
   {\rm BDI} & $(B_{n}, D_{i}\times B_{n-i})$ & $SO(2n+1)$ & $SO(2i)\times SO(2n-2i+1)$ &  $1\le i\le n$ \\ 
    & $(D_{n}, D_{i}\times D_{n-i})$ & $SO(2n)$ & $SO(2i)\times SO(2n-2i)$ &  $1\le i\le n-1$  \\
    {\rm DIII} & $(D_{n}, A_{n-1}\times T^{1})$ & $SO(2n)$ & $U(n)$ & \\
    {\rm CI} & $(C_{n}, A_{n-1}\times T^{1})$ & $Sp(n)$ & $U(n)$ &  \\
    {\rm CII} & $(C_{n},C_{i}\times C_{n-i})$ & $Sp(n)$ & $Sp(i)\times Sp(n-i)$ & $1\le i\le n-1$ \\ 
    {\rm EII} & $(E_{6}, A_{5}\times A_{1})$ & $E_{6}$ & $(SU(6)\times SU(2))/\mathbb{Z}_{2}$ & \\
    {\rm EIII} & $(E_{6}, D_{5}\times T^{1})$ & $E_{6}$ & $(Spin(10)\times T^{1})/\mathbb{Z}_{4}$ & \\
    {\rm EV} & $(E_{7},A_{7})$ & $E_{7}$ & $SU(8)/\mathbb{Z}_{2}$ & \\
    {\rm EVI} & $(E_{7},D_{6}\times A_{1})$ & $E_{7}$ & $(Spin(12)\times SU(2))/\mathbb{Z}_{2}$ & \\
    {\rm EVII} & $(E_{7},E_{6}\times T^{1})$ & $E_{7}$ & $(E_{6}\times T^{1})/\mathbb{Z}_{3}$ & \\
    {\rm EVIII} & $(E_{8},D_{8})$ & $E_{8}$ & $Ss(16)$ & \\
    {\rm EIX} & $(E_{8},E_{7}\times A_{1})$ & $E_{8}$ & $(E_{7}\times SU(2))/\mathbb{Z}_{2}$ & \\
   {\rm FI} & $(F_{4}, C_{3}\times A_{1})$ & $F_{4}$ & $(Sp(3)\times Sp(1))/\mathbb{Z}_{2}$ & \\
   {\rm FII} & $(F_{4}, B_{4})$ & $F_{4}$ & $Spin(9)$ & \\
   \hline
 \end{tabular}
\end{center}
 
Here, $n\ge 3$, $Ss(16)\simeq Spin(16)/\mathbb{Z}_{2}$ is the semi-spinor group (see~\cite[Theorem 2.1]{IT}), and we identify $A_{1}\simeq B_{1}\simeq C_{1}$, $D_{1}\simeq T^{1}$, $B_{2}\simeq C_{2}$, $D_{2}\simeq A_{1}\times A_{1}$ and $D_{3}\simeq A_{3}$.
\end{theorem}

We will often use the list in Theorem~\ref{thm:cart}, in particular, to state the main theorem, Theorem~\ref{main}.

\begin{remark}
It is important for our computation that any symmetric space $G/H$ from the list in Theorem~\ref{thm:cart} satisfies the property that $H$ is a maximal proper connected closed subgroup in $G$.
The respective root system embedding can be described by the Borel-de Siebenthal classification~\cite{BoSi}.
One can see from it that a homogeneous space $G/H$ (of maximal rank) is a symmetric space of type $1$ in the classification list by Cartan with inner involution if and only if the respective maximal subgroup $H$ in $G$ corresponds to removing a vertex in the (extended) Dynkin diagram of $G$ with multiplicity $m_{i}=1$ or $2$ (see Theorem~\ref{thm:BdS}).
\end{remark}

\section{Main theorem and its application}
\label{sect:3}

In this section, we state our main theorem. 
We also give its proof (modulo root system independence computations in later sections) and an application to the topology of the orbit spaces.

\subsection{Main theorem}
\label{sect:3.1}
If we assume $G$ is simply connected, then the standard $T$-action on the simply connected homogeneous space $G/H$ is equivariantly diffeomorphic to the following product of the homogeneous spaces (see e.g.~\cite[Section 2.2]{Ku10}):
\begin{align}
\label{splitting}
G/H\cong G_{1}/H_{1}\times \cdots \times G_{r}/H_{r},
\end{align}
where $G_{i}$ is a compact, simply connected simple Lie group, $H_{i}$ be its maximal rank subgroup and there is an isomorphism $T\cong T_{1}\times \cdots \times T_{r}$ such that $T_{i}\subset H_{i}\subset G_{i}$ is a maximal torus for $i=1,\ldots, r$. 
Moreover, we can easily show the following proposition. (Recall that a torus manifold is equipped with a half-dimensional torus action.)
\begin{proposition}
\label{split_computation}
For the splitting~\eqref{splitting}, let 
\[
S=\{i\mid G_i/H_i \text{ is not a torus manifold},\ 1\leq i\leq r\}.
\] 
Assume that $G/H$ is not a torus manifold, i.e., $S\neq\varnothing$. 
Then, there is the following equality:
\begin{align*}
k(G/H)=\min \{k(G_{i}/H_{i})\ |\ i\in S\}.
\end{align*}
\end{proposition} 

\begin{remark}
For the case of a product for torus manifolds, the independencies equate half dimensions of the respective manifolds.
Therefore, one has to replace minimum with the sum in the right hand side of the above formula.
\end{remark}

Proposition~\ref{split_computation} shows that to compute the independence of $G/H$,  it is enough to compute the independence of the simple factors $G_{i}/H_{i}$.
Any connected subgroup of $G$ fits into a saturated chain of maximal subgroups, where at every step the respective embedding is that of maximal proper subgroup.

The main theorem of this paper is as follows.

\begin{theorem}
\label{main}
Assume that $G$ is a simply connected, simple Lie group and $H$ be its maximal rank 
subgroup such that ${\rm rank}(G)={\rm rank}(H)=\dim T=n$.
Then, $k(G/H)=2,3$, or $n$.
If $n=3$, then $k(G/H)=3=n$ iff $(G,H)$ is one of the following pairs listed in Theorem~\ref{thm:cart}: 
\begin{align*}
(A_{3}, A_{2}\times T^{1}), (A_{3}, A_{1}\times A_{1}\times T^{1}), (B_{3},D_{3}), (C_{3}, C_{1}\times C_{2}).
\end{align*}
If $n>3$, then one and only one of the following holds:
\begin{enumerate}
\item $k(G/H)=n={\rm rank}(G)={\rm rank}(H)$ if and only if $(G,H)$ is one of the following types:
\begin{description}
\item[Type $A_{n}$] $(A_{n},A_{n-1}\times T^{1})$;
\item[Type $B_{n}$] $(B_{n}, D_{n})$.
\end{description}
\item $k(G/H)=3$ 
if and only if 
$(G,H)$ is one of the pairs appeared in Theorem~\ref{thm:cart} except for the above two cases, $(B_{n},D_{i}\times B_{n-i})$, $(C_{n},A_{n-1}\times T^{1})$ and $(F_{4}, C_{3}\times A_{1})$.
\item $k(G/H)=2$, otherwise (all of the other cases). 
\end{enumerate}
\end{theorem}

\begin{remark}
\label{about_torus_mainfold1}
Using the Borel-Hirzebruch result~\cite{BoHi}, we have that there are exactly two $Sp(n)$-invariant almost complex structures on $G/H=Sp(n)/(Sp(n-1)\times T^{1}) (\cong \mathbb{C}P^{2n-1})$.
Since one of them is the standard complex structure on $\mathbb{C}P^{2n-1}$, the $T^{n}$-action on $Sp(n)/(Sp(n-1)\times T^{1})$ extends to the $T^{2n-1}$-action.
Namely, $Sp(n)/(Sp(n-1)\times T^{1})$ is isomorphic to $SU(2n)/S(U(2n-1)\times U(1))$; therefore, this is a torus manifold.
Note that, by Theorem~\ref{main}, the homogeneous GKM manifold of type $(C_{n},C_{n-1}\times T^{1})$, i.e., $G/H=Sp(n)/(Sp(n-1)\times T^{1})$ with the $T^{n}$-action, satisfies $k(G/H)=2$.
This implies that a torus manifold structure on a given homogeneous manifold $G/H$ may exist when $k(G/H)\le 3$.
However, among the spaces $G/H$, since there is no case with $4\le k(G/H)\le n-1$ by Theorem~\ref{main}, Problem~\ref{Masuda-problem} is affirmatively solved for the maximal rank compact homogeneous spaces $G/H$.
\end{remark}

\begin{remark}
\label{about_torus_mainfold2}
We also note that, by~\cite{Ku09}, the homogeneous space $G/H$ can be a torus manifold only if it is isomorphic to $A_{n}/(A_{n-1}\times T^{1})$ or $B_{n}/D_{n}$. 
Therefore, it is easy to check that if $k(G/H)\le 3$ and $G/H$ is a torus manifold, then its GKM graph must coincide with that of $A_{n}/(A_{n-1}\times T^{1})$ or $B_{n}/D_{n}$ equipped with the non-standard connection (see~\cite{GZ, Ku19} for the notion of connection).
If $k(G/H)=3$, then the connection is unique by~\cite{GZ}.
Moreover, by using~\cite{Ku19}, if such a GKM graph extends to that of $A_{n}/(A_{n-1}\times T^{1})$ or $B_{n}/D_{n}$, then the connection must be the standard one.
Consequently, if $n\ge 4$, there is no $G/H$ with $k(G/H)=3$ that is isomorphic to a torus manifold.
We also remark that, in the case $G/H=Sp(n)/(Sp(n-1)\times T^{1})$ considered in Remark~\ref{about_torus_mainfold1}, the standard complex structure corresponds to the standard connection on its GKM graph, while the non-standard almost complex structure corresponds to the non-standard connection.
This phenomenon can occur precisely because $k(G/H)=2$.
\end{remark}

\subsection{Proof of the main theorem}
\label{sect:3.2}
Here, we prove Theorem~\ref{main}, assuming several facts that will be established in later sections.
The following key lemma significantly reduces the number of cases we need to check.

\begin{lemma}\label{auxiliary lemma}
Under the hypothesis in Theorem~\ref{main}, 
if $k(G/H)\ge 3$, then  $G/H$ is a symmetric space given in Theorem~\ref{thm:cart}.
\end{lemma}
\begin{proof}
Suppose that $G/H$ is not a symmetric space, i.e., $[\mathfrak{m},\mathfrak{m}]\not\subset \mathfrak{h}$ by Proposition~\ref{pr:symm}.
Then its complexification also satisfies that $[\mathfrak{m}_{\mathbb{C}},\mathfrak{m}_{\mathbb{C}}]\not\subset \mathfrak{h}_{\mathbb{C}}$. For the decomposition
\begin{align*}
\mathfrak{m}_{\mathbb{C}}=\bigoplus_{\alpha\in \Delta_{G,H}}\mathfrak{g}_{\alpha}=\bigoplus_{\alpha\in \Delta_{G,H}^{+}}(\mathfrak{g}_{\alpha}\oplus \mathfrak{g}_{-\alpha}),
\end{align*}
where we regard $\mathfrak{g}_{\alpha}$ as the one-dimensional complex root space corresponding to $\alpha$. Therefore, 
there exist roots $a,b\in \Delta_{G,H}$ such that $a+b$ is in $\Delta_{G,H}$, by the property
\begin{equation}\label{eq:lbprop}
[\mathfrak{g}_{a},\mathfrak{g}_{b}]=\mathfrak{g}_{a+b},
\end{equation}
of a Lie algebra (see~\cite[Ch III, Theorem 4.3 (iv)]{He}).
This implies that $\eps_{1}a,\eps_{2}b, \eps_{3}(a+b)\in\Delta^{+}_{G,H}$ holds for some $\eps_{1},\eps_{2},\eps_{3}\in\lb \pm 1 \rb$. 
Since changing signs does not affect linear dependence, these three positive roots are linearly dependent.
Therefore, $\Delta^{+}_{G,H}$ is not $3$-independent, i.e. $k(G/H)=2$.
This shows that if $k(G/H)\ge 3$, then $G/H$ is a symmetric space.
\end{proof}

Now we may prove Theorem~\ref{main}.
\begin{proof}[Proof of Theorem~\ref{main}]
Using Lemma~\ref{auxiliary lemma}, 
if $G/H$ is not a symmetric space in Theorem~\ref{thm:cart}, then $k(G/H)=2$.
When $G/H$ is a symmetric space in Theorem~\ref{thm:cart},
in Sections~\ref{sect:5}--\ref{sect:7} we find $k(G/H)$ in all cases except for $G$ of type $F_{4},E_{6},E_{7},E_{8}$ using weighted signed graphs.
In Sections~\ref{sect:8} and~\ref{sect:9}, we find $k(G/H)$ for $G$ of type $F_{4},E_{6},E_{7},E_{8}$ by another method.
The obtained values of $k(G/H)$ belong to $\lb 2,3,n\rb$, proving the first claim of the theorem.
Assume that $k(G/H)\geq 3$ holds.
By comparing the above mentioned computation of $k(G/H)$ with the list of symmetric spaces from Theorem~\ref{thm:cart}, we exclude the pairs not satisfying this assumption, i.e. $(B_{n},D_{i}\times B_{n-i})$ for $1\le i<n$, $(C_{n},A_{n-1}\times T^{1})$ and $(F_{4}, C_{3}\times A_{1})$.
In particular, this proves the claim for any $(G,H)$ satisfying $n=3$.
For $n>3$, the above computation shows that $k(G,H)=n$ holds iff $(G,H)$ is either $(A_{n},A_{n-1}\times T^{1})$ or $(B_{n}, D_{n})$.
Furthermore, for $n>3$ the cases $(1)$--$(3)$ from the claim of theorem are pairwise disjoint.
This finishes the proof.
\end{proof}

As a corollary of Theorem~\ref{main}, we have the following description of GKM$_{4}$ homogeneous spaces by using the independence of tangential weights (also see the classification of the homogeneous torus manifolds in~\cite{Ku10}).
This gives the partial answer to Problem~\ref{Masuda-problem}. 

\begin{corollary}
If $G/H$ is a simply connected homogeneous GKM manifold with $k(G/H)\ge 4$, i.e., GKM$_{4}$ homogeneous space, then $G/H$ is a torus manifold which is equivariantly diffeomorphic to a product of $\mathbb{C}P^{n}=SU(n+1)/S(U(n)\times U(1))$ or $S^{2n}=SO(2n+1)/SO(2n)$.
\end{corollary}

\subsection{An application to the topology of orbit spaces of homogeneous spaces}
\label{sect:3.3}
On the other hand, Ayzenberg and Masuda in~\cite{AM} studied topological properties of GKM$_{k}$ manifolds and proved the following theorem concerning the vanishing of homology of their orbit spaces.

\begin{theorem}\cite[Theorem 2]{AM}\label{thm:AM}
If $M$ is a GKM$_{k}$ manifold with $T$-action satisfying $H^{odd}(M;\Bbbk)=0$ and either $\Bbbk=\Q$, or $\Bbbk=\Z$ and the $T$-action on $M$ has only connected stabilizers, then the orbit space $M/T$ is $(k+1)$-acyclic, i.e. $\widetilde{H}^{i}(M/T;\Bbbk)=0$ for any $i\leq k+1$.
\end{theorem}

In order to apply Theorem~\ref{thm:AM} for $3$-independent cases (i.e., GKM$_{3}$ cases) in Theorem~\ref{main},
in what follows we need to know $H^{odd}(G/H,\Bbbk)=0$ for particular $\Bbbk$ coefficients.
If $\Bbbk$ is a field of $\chr \Bbbk=0$, this always holds~\cite[Prop. 30.2]{B53}.
Furthermore, one can show that for $G$ of type $A_{n}$ the homogeneous GKM manifold $G/H$ has only connected stabilizers of the respective $T$-action.
We have the following corollary:

\begin{corollary}\label{cor:acyc}
If $G/H$ is a symmetric GKM$_{3}$ space (see Theorem~\ref{thm:cart} and (1), (2) of Theorem~\ref{main})
with the following coefficient $\Bbbk$ for the reduced homology, then $\widetilde{H}^{i}(T\backslash G/H;\Bbbk)=0$ for $i\le 4$.
Here, $\Bbbk=\Z$ or $\Q$ for $G$ of type $A_{n}$.
In the remaining cases, $\Bbbk=\Q$.
\end{corollary}

\begin{example}\label{ex:origra}
Consider the homogeneous manifold
\[
Gr^{+}(2,n)=SO(n)/(SO(n-2)\times SO(2))=Q_{n},
\]
with the natural effective $T$-action.
Here, if $n=2r$ then $T=T^{r}/{\pm 1}$, and if $n=2r+1$ then $T=T^{r}$.
This oriented Grassmannian is a connected double cover of the real Grassmannian $Gr(2,n)$.
This implies $Gr^{+}(2,3)=S^{2}$.
It is not difficult to show that $Gr^+(2,4)$ is homeomorphic to $S^2\times S^2$.
(Furthermore, $Gr^+(2,6)$ is homeomorphic to the complex Grassmannian $Gr_{\C}(2,4)$.)
The orbit space is homeomorphic to $\Sigma^{\lfloor n/2\rfloor} \C P^{\lfloor n/2\rfloor-2}$ by~\cite{S19}.
In particular, the orbit space is ${\lfloor n/2\rfloor}$-connected (where the suspension of an empty set is by definition a point).
In particular, it is contractible for $n=3,4$.
This homogeneous manifold has type $D_{r}/(D_{r-1}\times T^{1})$ for $n=2r$ and $B_{r}/(B_{r-1}\times T^1)$ for $n=2r+1$, where $r\geq 2$.
Here, $B_1=A_1$, $D_1=T^{1}$, and $D_2=(A_1)^2$.
Since $Gr^{+}(2,2r)=Q_{2r}$ and $Gr^{+}(2,2r+1)=Q_{2r+1}$, 
It is known that $H^{odd}(Q_{n};\mathbb{Z})=0$.
Therefore, Theorem~\ref{thm:AM} implies that the orbit space is $4$-acyclic (over $\Q$) for $n=2r$ and $3$-acyclic for $n=2r+1$.
Therefore, our result agrees with previous known results, in particular, in dimensions $n=6,7,8$.
\end{example}

\begin{example}\label{ex:comgra}
Consider the homogeneous manifold
\[
Gr(2,5)=SU(5)/S(U(2)\times U(3)),
\]
with the natural effective $T^{4}$-action.
This is the complex Grassmannian of two-planes in $\C^5$.
The fixed-point data at the identity is a root system of type $A_{4}/(A_{1}\times A_{2}\times T^{1})$.
By Theorems~\ref{main} and Corollary~\ref{cor:acyc},
this GKM-manifold has $4$-acyclic (over $\Q$) orbit space $Gr(2,5)/T^4$.
On the other hand, by~\cite{BT} or~\cite{S21},
\[
\widetilde{H}_{i}(Gr(2,5)/T^4;\Z)\simeq
\begin{cases}
\Z/2\Z,\ i=5,\\
\Z,\ i=8,\\
0,\ \mbox{else}.
\end{cases}
\]
This agrees with the above acyclicity claim.
\end{example}


\begin{remark}
The second named author finds the negative answer to Problem~\ref{Masuda-problem} for the abstract GKM graphs which are the combinatorial counterpart of GKM manifolds in~\cite{So}. 
Furthermore, this problem is positively solved for GKM graphs of GKM manifolds that are either Hamiltonian or have complexity $1$ by~\cite{GS}.
However, the geometric problem is still open. 
\end{remark}

\section{The weighted signed graph associated with positive root systems}
\label{sect:4}

In this section, we first recall the signed graph, introduced by Harary~\cite{Ha} and have been extensively studied by Zaslavsky from a matroid theoretic viewpoint 
(see e.g.~\cite{Za} from one of his series of papers)
and also studied by Fern-Gordon-Leasure-Pronchik or Fried-Gerek-Gordon-Perun\v{c}i\'{c} for an automorphism of matroid induced from the root systems of classical types and $F_{4}$, see~\cite{FGLP, FGGP}.

Let $\Gamma:=(V,E)$ be a graph (with loops and multiple edges allowed).
A {\it signed graph} is a graph $\Gamma$ in which each edge $e\in E$ is labeled with a $+$ or $-$ sign, say $s(e)$.
In this paper, we mainly consider the following three types of signed graphs (see Figure~\ref{examples_signed_graphs}):

\begin{description}
\item[Slim] the complete graph of $n$ vertices such that all edges have the same sign, say $K_{n}$;
\item[Puff] the complete graph of $n$ vertices 
such that every two vertices are connected by exactly two edges (i.e., a multiple-edge) whose signs are $+$ and $-$ respectively, say $\pm K_{n}$;
\item[Superpuff] the signed graph $\pm K_{n}$ with $n$ loops in total, one on each vertex, say $K_{n}^{\circ}$.
\end{description}

\begin{figure}[H]
\begin{tikzpicture}
\begin{scope}[xscale=0.5, yscale=0.5]
\fill(-5,-2)circle (5pt);
\node[left] at (-5,-2) {$1$};
\fill(-1,-2)circle (5pt);
\node[right] at (-1,-2) {$2$};
\fill(-1,2)circle (5pt);
\node[right] at (-1,2) {$3$};
\fill(-5,2)circle (5pt);
\node[left] at (-5,2) {$4$};

\draw[thick] (-5,-2)--(-5,2);
\draw[thick] (-5,-2)--(-1,-2);
\draw[thick] (-5,-2)--(-1,2);
\draw[thick] (-5,2)--(-1,-2);
\draw[thick] (-5,2)--(-1,2);
\draw[thick] (-1,2)--(-1,-2);

\fill(1,-2)circle (5pt);
\node[left] at (1,-2) {$1$};
\fill(5,-2)circle (5pt);
\node[right] at (5,-2) {$2$};
\fill(5,2)circle (5pt);
\node[right] at (5,2) {$3$};
\fill(1,2)circle (5pt);
\node[left] at (1,2) {$4$};

\draw[dashed](1,-2)to[out=15,in=165](5,-2);
\draw[thick](1,-2)to[out=-15,in=-165](5,-2);
\draw[thick](1,2)to[out=15,in=165](5,2);
\draw[dashed](1,2)to[out=-15,in=-165](5,2);
\draw[thick](5,-2)to[out=75,in=-75](5,2);
\draw[dashed](5,-2)to[out=105,in=-105](5,2);
\draw[thick](1,-2)to[out=105,in=-105](1,2);
\draw[dashed](1,-2)to[out=75,in=-75](1,2);
\draw[thick](1,-2)to[out=60,in=-150](5,2);
\draw[dashed](1,-2)to[out=30,in=-105](5,2);
\draw[thick](1,2)to[out=-30,in=120](5,-2);
\draw[dashed](1,2)to[out=-60,in=150](5,-2);

\fill(8,-2)circle (5pt);
\node[label={200:$1$}] at (8,-2) {};
\fill(12,-2)circle (5pt);
\node[label={290:$2$}] at (12,-2) {};
\fill(12,2)circle (5pt);
\node[label={20:$3$}] at (12,2) {};
\fill(8,2)circle (5pt);
\node[label={160:$4$}] at (8,2) {};

\draw[dashed](8,-2)to[out=15,in=165](12,-2);
\draw[thick](8,-2)to[out=-15,in=-165](12,-2);
\draw[thick](8,2)to[out=15,in=165](12,2);
\draw[dashed](8,2)to[out=-15,in=-165](12,2);
\draw[thick](12,-2)to[out=75,in=-75](12,2);
\draw[dashed](12,-2)to[out=105,in=-105](12,2);
\draw[thick](8,-2)to[out=105,in=-105](8,2);
\draw[dashed](8,-2)to[out=75,in=-75](8,2);
\draw[thick](8,-2)to[out=60,in=-150](12,2);
\draw[dashed](8,-2)to[out=30,in=-105](12,2);
\draw[thick](8,2)to[out=-30,in=120](12,-2);
\draw[dashed](8,2)to[out=-60,in=150](12,-2);

\draw[thick](7.5,-2.5) circle [x radius=0.7cm, y radius=0.7cm, rotate=30];
\draw[thick](7.5,2.5) circle [x radius=0.7cm, y radius=0.7cm, rotate=30];
\draw[thick](12.5,-2.5) circle [x radius=0.7cm, y radius=0.7cm, rotate=30];
\draw[thick](12.5,2.5) circle [x radius=0.7cm, y radius=0.7cm, rotate=30];

\end{scope}
\end{tikzpicture}
\caption{The signed graph $K_{4}$, $\pm K_{4}$ and $K_{4}^{\circ}$ (from left). Edges with $s(e)=-$ (resp. $s(e)=+$) are drawn as solid (resp. dashed) edges}
\label{examples_signed_graphs}
\end{figure}
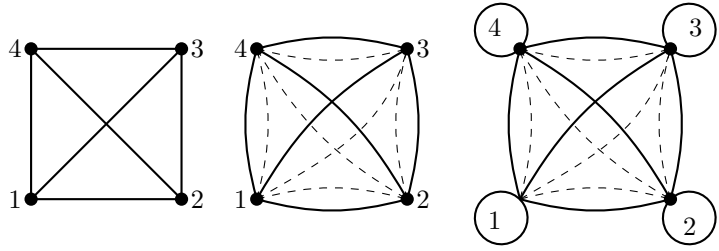

Moreover, in this paper, we give a weight on each vertex by a vector in $\mathbb{R}^{n}$, say $w_{V}\colon V\to \mathbb{R}^{n}$ (see Figure~\ref{example_ABDG}).
In this situation, we may also define weights on edges and loops, say $w_{E}\colon E\to \mathbb{R}^{n}$, by 
\begin{itemize}
\item if $l\in E$ is a loop on the vertex $v\in V$, then $w_{E}(l):=w_{V}(v)$;
\item if $e\in E$ is an edge connecting two vertices $p,q\in V$, then for some non-zero $r^{+}_{pq}, r^{-}_{pq}\in \mathbb{R}$, $w_{E}(e):=r^{+}_{pq}(w_{V}(p)-w_{V}(q))$ if $s(e)=+$ and $w_{E}(e):=r^{-}_{pq}(w_{V}(p)+w_{V}(q))$ if $s(e)=-$. 
\end{itemize}
We call such a signed graph equipped with a weight $w_{V}, w_{E}$ a {\it weighted signed graph}.

The positive roots of certain compact Lie groups can be regarded as a weighted signed graph as follows, where $e_{i}\in \mathbb{R}^{N}$ is the standard basis for $i=1,\ldots, N$ and the edge $e=ij\in E$ connecting vertices $i$, $j\in V$.
\begin{itemize}
\item for $A_{n}$ ($n\ge 1$), the weighted signed graph is $K_{n+1}$ with the weights $w_{V}(i):=e_{i}$ and $r^{+}_{ij}=-r^{+}_{ji}=1$, where $i,j\in V=[n+1]$ and $1\le i<j\le n+1$;
\item for $B_{n}$ ($n\ge 2$), the weighted signed graph is $K_{n}^{\circ}$ with the weights $w_{V}(i):=e_{i}$ and $r_{ij}^{+}=-r^{+}_{ji}=r_{ij}^{-}=r_{ji}^{-}=1$, where $i, j\in V=[n]$ and $1\le i<j\le n$;
\item for $C_{n}$($n\ge 2$), the weighted signed graph is $K_{n}^{\circ}$ with the weights $w_{V}(i):=2e_{i}$ and $r_{ij}^{+}=-r^{+}_{ji}=r_{ij}^{-}=r_{ji}^{-}=\frac{1}{2}$, where $i\in V=[n]$ and $1\le i<j\le n$;
\item for $D_{n}$ ($n\ge 4$), the weighted signed graph is $\pm K_{n}$ with the weights $w_{V}(i):=e_{i}$ and $r_{ij}^{+}=-r^{+}_{ji}=r_{ij}^{-}=r_{ji}^{-}=1$, where $i\in V=[n]$;
\end{itemize}

\begin{figure}[H]
\centering
\begin{subfigure}{0.2\textwidth}
\begin{tikzpicture}
\begin{scope}[xscale=0.3, yscale=0.3]
\coordinate (A) at (4,0);
\coordinate (B) at ({4*cos(2*pi/3 r)},{4*sin(2*pi/3 r)});
\coordinate (C) at ({4*cos(4*pi/3 r)},{4*sin(4*pi/3 r)});
\fill(A) circle (5pt);
\node[above] at (A) {$e_{1}$};
\fill(B)circle (5pt);
\node[left] at (B) {$e_{2}$};
\fill(C)circle (5pt);
\node[left] at (C) {$e_{3}$};
\draw (A)--(B)--(C)--cycle;
\end{scope}
\end{tikzpicture}
\caption*{$A_{2}$}
\label{example_weighted_vertices_A}
\end{subfigure}\hspace{1em}
\begin{subfigure}{0.2\textwidth}
\begin{tikzpicture}
\begin{scope}[xscale=0.3, yscale=0.3]

\coordinate (A) at (4,0);
\coordinate (B) at ({4*cos(2*pi/3 r)},{4*sin(2*pi/3 r)});
\coordinate (C) at ({4*cos(4*pi/3 r)},{4*sin(4*pi/3 r)});

\fill(A) circle (5pt);
\fill(B)circle (5pt);
\fill(C)circle (5pt);

\draw[dashed](A)to[out=115,in=0](B);
\draw[thick](A)to[out=170,in=-55](B);
\draw[thick](A)to[out=190,in=55](C);
\draw[dashed](A)to[out=-115,in=0](C);
\draw[thick](B)to[out=295,in=65](C);
\draw[dashed](B)to[out=245,in=115](C);

\draw[thick](4.7,0) circle [x radius=0.7cm, y radius=0.7cm];
\node[above] at (4.7,0.7) {$e_{1}$};
\draw[thick](-2.5,4) circle [x radius=0.7cm, y radius=0.7cm];
\node[left] at (-3.2,4) {$e_{2}$};
\draw[thick](-2.5,-4) circle [x radius=0.7cm, y radius=0.7cm];
\node[left] at (-3.2,-4) {$e_{3}$};
\end{scope}
\end{tikzpicture}
\caption*{$B_{3}$}
\label{example_weighted_vertices_B}
\end{subfigure}\hspace{1em}
\begin{subfigure}{0.2\textwidth}
\begin{tikzpicture}
\begin{scope}[xscale=0.5, yscale=0.5]

\fill(8,-2)circle (5pt);
\node[label={200:$1$}] at (8,-2) {};
\fill(12,-2)circle (5pt);
\node[label={290:$2$}] at (12,-2) {};
\fill(12,2)circle (5pt);
\node[label={20:$3$}] at (12,2) {};
\fill(8,2)circle (5pt);
\node[label={160:$4$}] at (8,2) {};

\draw[dashed](8,-2)to[out=15,in=165](12,-2);
\draw[thick](8,-2)to[out=-15,in=-165](12,-2);
\draw[thick](8,2)to[out=15,in=165](12,2);
\draw[dashed](8,2)to[out=-15,in=-165](12,2);
\draw[thick](12,-2)to[out=75,in=-75](12,2);
\draw[dashed](12,-2)to[out=105,in=-105](12,2);
\draw[thick](8,-2)to[out=105,in=-105](8,2);
\draw[dashed](8,-2)to[out=75,in=-75](8,2);
\node[right] at (12,0) {$a_{23}^{-}$};
\node[left] at (12,0) {$a_{23}^{+}$};
\draw[thick](8,-2)to[out=60,in=-150](12,2);
\draw[dashed](8,-2)to[out=30,in=-105](12,2);
\draw[thick](8,2)to[out=-30,in=120](12,-2);
\draw[dashed](8,2)to[out=-60,in=150](12,-2);

\draw[thick](7.5,-2.5) circle [x radius=0.7cm, y radius=0.7cm, rotate=30];
\node[label={300:$2e_{1}$}] at (8,-2) {};
\draw[thick](7.5,2.5) circle [x radius=0.7cm, y radius=0.7cm, rotate=30];
\node[label={70:$2e_{4}$}] at (8,2) {};
\draw[thick](12.5,-2.5) circle [x radius=0.7cm, y radius=0.7cm, rotate=30];
\node[label={250:$2e_{2}$}] at (12,-2) {};
\draw[thick](12.5,2.5) circle [x radius=0.7cm, y radius=0.7cm, rotate=30];
\node[label={110:$2e_{3}$}] at (12,2) {};

\end{scope}
\end{tikzpicture}
\caption*{$C_{4}$}
\label{example_weighted_vertices_C}
\end{subfigure}\hspace{1em}
\begin{subfigure}{0.2\textwidth}
\begin{tikzpicture}
\begin{scope}[xscale=0.5, yscale=0.5]

\fill(8,-2)circle (5pt);
\node[left] at (8,-2) {$e_{1}$};
\fill(12,-2)circle (5pt);
\node[right] at (12,-2) {$e_{2}$};
\fill(8,2)circle (5pt);
\node[left] at (8,2) {$e_{4}$};
\fill(12,2)circle (5pt);
\node[right] at (12,2) {$e_{3}$};

\draw[dashed](8,-2)to[out=15,in=165](12,-2);
\draw[thick](8,-2)to[out=-15,in=-165](12,-2);
\draw[thick](8,2)to[out=15,in=165](12,2);
\draw[dashed](8,2)to[out=-15,in=-165](12,2);
\draw[thick](12,-2)to[out=75,in=-75](12,2);
\draw[dashed](12,-2)to[out=105,in=-105](12,2);
\draw[thick](8,-2)to[out=105,in=-105](8,2);
\draw[dashed](8,-2)to[out=75,in=-75](8,2);
\draw[thick](8,-2)to[out=60,in=-150](12,2);
\draw[dashed](8,-2)to[out=30,in=-105](12,2);
\draw[thick](8,2)to[out=-30,in=120](12,-2);
\draw[dashed](8,2)to[out=-60,in=150](12,-2);

\end{scope}
\end{tikzpicture}
\caption*{$D_{4}$}
\label{example_weighted_vertices_D}
\end{subfigure}\hspace{1em}
\caption{The weighted signed graphs corresponding to several types.
For the weighted signed graph $C_{4}$, the weights on the vertex can be known from the weights on loops, and edges are 
$a_{ij}^{-}=e_{i}-e_{j}$ and $a_{ij}^{+}=e_{i}+e_{j}$ for $1\le i<j\le 4$}
\label{example_ABDG}
\end{figure}
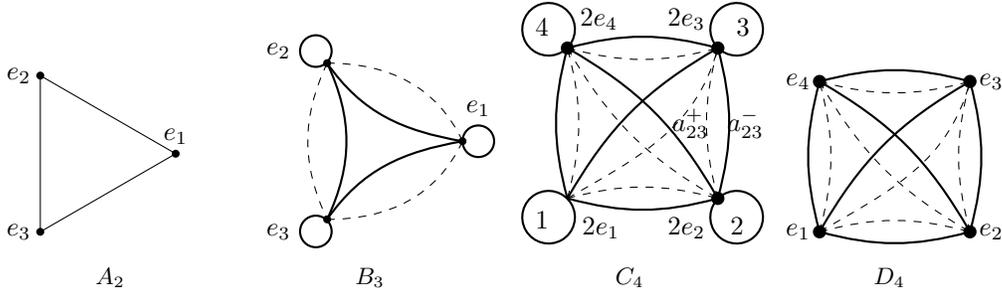

By definition of the weighted signed graph and the well-known positive root systems (see~\cite[Chapter V, 6.26]{MiTo}), it is easy to check the following lemma:
\begin{lemma}
\label{correspondence}
For each root system of types $A$--$D$, 
 the weights assigned to the edges or loops of the weighted signed graph correspond bijectively to certain positive roots of the respective root system.
\end{lemma}

For the maximal rank subgroup $H\subset G$, the positive sub-roots of $H$ may be regarded as the subgraph of the weighted signed graph $\Gamma_{G}$.
In particular, the set $\Delta_{G,H}/\{\pm 1\}$ corresponds to the graph, say $\Gamma(G,H)$, which consists of the vertices of $\Gamma_{G}$ and the edges of $\Gamma_{G}\setminus \Gamma_{H}$ denoted by $E_{G,H}$. 
So to compute $k(G/H)$ it is enough to check the independence of the
weights of the edges $E_{G,H}$.

\begin{lemma}\label{lm:23dep}
Let $G$ be a simple Lie group of type $A$--$D$, and $H$ be
its proper maximal connected (closed) maximal rank subgroup.
Then one has the following:
\begin{enumerate}
\item $k(G/H)=2$ if and only if the weighted signed graph $\Gamma(G,H)$ has a subgraph (with weighted vertices) appeared in Figure~\ref{example_weighted_vertex}, where $i<j<k$.
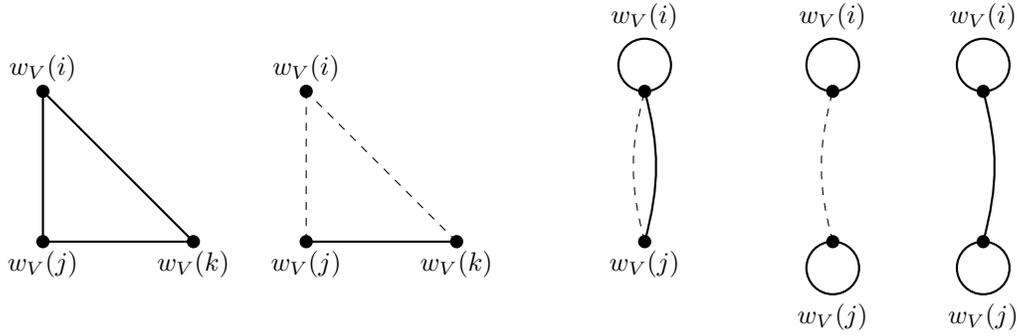
\begin{figure}[H]
\begin{tikzpicture}
\begin{scope}[xscale=0.5, yscale=0.5]

\fill(-2,-2)circle (5pt);
\fill(-2,2)circle (5pt);
\fill(2,-2)circle (5pt);
\draw[thick](-2,-2)--(-2,2);
\node[above] at (-2,2) {$w_{V}(i)$};
\draw[thick](2,-2)--(-2,2);
\node[below] at (-2,-2) {$w_{V}(j)$};
\draw[thick](-2,-2)--(2,-2);
\node[below] at (2,-2) {$w_{V}(k)$};

\fill(5,2)circle (5pt);
\fill(5,-2)circle (5pt);
\fill(9,-2)circle (5pt);
\draw[dashed](5,-2)--(5,2);
\node[above] at (5,2) {$w_{V}(i)$};
\draw[dashed](9,-2)--(5,2);
\node[below] at (5,-2) {$w_{V}(j)$};
\draw[thick](9,-2)--(5,-2);
\node[below] at (9,-2) {$w_{V}(k)$};

\fill(14,2)circle (5pt);
\fill(14,-2)circle (5pt);
\draw[thick](14,-2)to[out=75,in=-75](14,2);
\draw[dashed](14,-2)to[out=105,in=-105](14,2);
\draw[thick](14,2.7) circle [x radius=0.7cm, y radius=0.7cm];
\node[above] at (14,3.4) {$w_{V}(i)$}; 
\node[below] at (14,-2) {$w_{V}(j)$};

\fill(19,2)circle (5pt);
\fill(19,-2)circle (5pt);
\draw[dashed](19,-2)to[out=105,in=-105](19,2);
\draw[thick](19,2.7) circle [x radius=0.7cm, y radius=0.7cm];
\node[above] at (19,3.4) {$w_{V}(i)$};
\draw[thick](19,-2.7) circle [x radius=0.7cm, y radius=0.7cm];
\node[below] at (19,-3.4) {$w_{V}(j)$};

\fill(23,2)circle (5pt);
\fill(23,-2)circle (5pt);
\draw[thick](23,-2)to[out=75,in=-75](23,2);
\draw[thick](23,2.7) circle [x radius=0.7cm, y radius=0.7cm];
\node[above] at (23,3.4) {$w_{V}(i)$};
\draw[thick](23,-2.7) circle [x radius=0.7cm, y radius=0.7cm];
\node[below] at (23,-3.4) {$w_{V}(j)$};

\end{scope}
\end{tikzpicture}
\caption{$2$-independent subgraphs} 
\label{example_weighted_vertex}
\end{figure}
\item $k(G/H)=3$ if the weighted signed graph $\Gamma(G,H)$ has neither of the subgraphs (with weighted vertices) depicted in Figure~\ref{example_weighted_vertex} but has a weighted subgraph appeared in  Figure~\ref{example_weighted_vertex2}, where $i<j<k<l$.
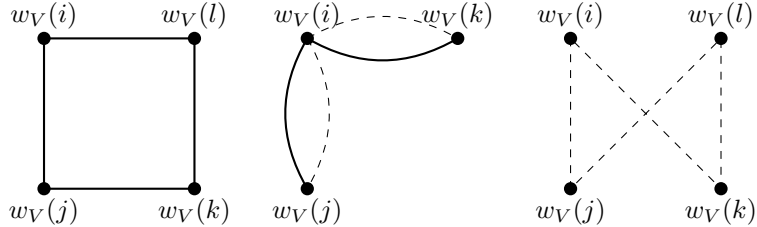
\begin{figure}[H]
\begin{tikzpicture}
\begin{scope}[xscale=0.5, yscale=0.5]

\fill(-2,-2)circle (5pt);
\fill(-2,2)circle (5pt);
\fill(2,-2)circle (5pt);
\fill(2,2)circle (5pt);
\draw[thick](-2,-2)--(-2,2)--(2,2)--(2,-2)--cycle;
\node[above] at (-2,2) {$w_{V}(i)$};
\node[below] at (-2,-2) {$w_{V}(j)$};
\node[below] at (2,-2) {$w_{V}(k)$};
\node[above] at (2,2) {$w_{V}(l)$};

\coordinate (A) at (5,2);
\coordinate (B) at (5,-2);
\coordinate (C) at (9,2);
\fill(A)circle (5pt);
\fill(B)circle (5pt);
\fill(C)circle (5pt);
\draw[thick](A)to[out=240,in=120](B);
\node[above] at (A) {$w_{V}(i)$};
\draw[dashed](A)to[out=300,in=60](B);
\node[below] at (B) {$w_{V}(j)$};
\draw[thick](A)to[out=330,in=210](C);
\node[above] at (C) {$w_{V}(k)$};
\draw[dashed](A)to[out=30,in=150](C);

\coordinate (a) at (12,2);
\coordinate (b) at (12,-2);
\coordinate (c) at (16,-2);
\coordinate (d) at (16,2);
\fill(a)circle (5pt);
\fill(b)circle (5pt);
\fill(c)circle (5pt);
\fill(d)circle (5pt);

\draw[dashed] (a)--(b)--(d)--(c)--(a);

\node[above] at (a) {$w_{V}(i)$};
\node[below] at (b) {$w_{V}(j)$};
\node[below] at (c) {$w_{V}(k)$};
\node[above] at (d) {$w_{V}(l)$};

\end{scope}
\end{tikzpicture}
\caption{$3$-independent subgraphs}
\label{example_weighted_vertex2}
\end{figure}
\end{enumerate}
\end{lemma}
\begin{proof}
For $\Gamma(G,H)$, since the statement $(2)$ is straightforward, we only prove the statement $(1)$.
If there is a subgraph which is drawn in Figure~\ref{example_weighted_vertex}, then by the definition of weighted signed graphs, $\Gamma(G,H)$ is $2$-independent, i.e., $k(G/H)=2$. 
 
Assume that $k(G/H)=2$, i.e., there are three edges (or loops) that are 2-independent in $\Gamma(G,H)$.
Let $\Gamma'$ be a subgraph in $\Gamma(G,H)$ consisting of such three edges (or loops). 
We first claim that $\Gamma'$ has two or three vertices.
It is easy to verify that a subgraph in $\Gamma(G,H)$ with three edges (or loops) and four or more vertices is one of the following graphs: 
\begin{itemize}
\item three disjoint edges (six vertices);
\item two disjoint edges and a separate loop (five vertices);
\item two connected edges with 3 vertices (i.e., V shape) and a separate edge (five vertices);
\item two connected edges with 3 vertices (i.e., V shape) and a separate loop (four vertices);
\item one edge and two separate loops (four vertices);
\item one multiple edge and a separate edge (four vertices);
\item three connected edges with 4 vertices (i.e., N shape)
\end{itemize}
By the definition of weighted signed graphs, all weights of these graphs correspond to linearly independent vectors. 
Hence, there are no subgraphs of three 2-independent edges (or loops) with 
four or more vertices.

If $\Gamma'$ consists of three vertices without loops, then by the definition of weighted signed graph there are the following four types in Figure~\ref{three_vertices} for some edges $i<j<k$.
\begin{figure}[H]
\begin{tikzpicture}
\begin{scope}[xscale=0.4, yscale=0.4]

\fill(-2,-2)circle (5pt);
\fill(-2,2)circle (5pt);
\fill(2,-2)circle (5pt);
\draw[thick](-2,-2)--(-2,2);
\node[above] at (-2,2) {$w_{V}(i)$};
\draw[thick](2,-2)--(-2,2);
\node[below] at (-2,-2) {$w_{V}(j)$};
\draw[thick](-2,-2)--(2,-2);
\node[below] at (2,-2) {$w_{V}(k)$};

\fill(5,2)circle (5pt);
\fill(5,-2)circle (5pt);
\fill(9,-2)circle (5pt);
\draw[dashed](5,-2)--(5,2);
\node[above] at (5,2) {$w_{V}(i)$};
\draw[dashed](9,-2)--(5,2);
\node[below] at (5,-2) {$w_{V}(j)$};
\draw[thick](9,-2)--(5,-2);
\node[below] at (9,-2) {$w_{V}(k)$};

\fill(12,-2)circle (5pt);
\fill(12,2)circle (5pt);
\fill(16,-2)circle (5pt);
\draw[thick](12,-2)--(12,2);
\node[above] at (12,2) {$w_{V}(i)$};
\draw[thick](16,-2)--(12,2);
\node[below] at (12,-2) {$w_{V}(j)$};
\draw[dashed](12,-2)--(16,-2);
\node[below] at (16,-2) {$w_{V}(k)$};

\fill(19,-2)circle (5pt);
\fill(19,2)circle (5pt);
\fill(23,-2)circle (5pt);
\draw[dashed](19,-2)--(19,2);
\node[above] at (19,2) {$w_{V}(i)$};
\draw[dashed](23,-2)--(19,2);
\node[below] at (19,-2) {$w_{V}(j)$};
\draw[dashed](19,-2)--(23,-2);
\node[below] at (23,-2) {$w_{V}(k)$};

\end{scope}
\end{tikzpicture}
\caption{Forbidden subgraphs on $3$ vertices without loops for $3$-independence}
\label{three_vertices} 
\end{figure}

For the left two cases above, they satisfy that 
\begin{align*}
& w_{E}(ij)+w_{E}(jk)=(w_{V}(i)-w_{V}(j))+(w_{V}(j)-w_{V}(k))=w_{V}(i)-w_{V}(k)=w_{E}(ik), \\
& w_{E}(ik)+w_{E}(jk)=(w_{V}(i)+w_{V}(k))+(w_{V}(j)-w_{V}(k))=w_{V}(i)+w_{V}(j)=w_{E}(ij).
\end{align*}
Thus, they satisfy the $2$-independence.
However, for the right two cases above, the weights 
\begin{align*}
& \{w_{E}(ij), w_{E}(jk), w_{E}(ik)\}=\{w_{V}(i)-w_{V}(j), w_{V}(j)+w_{V}(k), w_{V}(i)-w_{V}(k)\}, \\
& \{w_{E}(ij), w_{E}(jk), w_{E}(ik)\}=\{w_{V}(i)+w_{V}(j), w_{V}(j)+w_{V}(k), w_{V}(i)+w_{V}(k)\} 
\end{align*}
are linearly independent. 
Therefore, they are not $2$-independent.
Consequently, if $\Gamma'$ consists of three vertices without loops, then it must be the 1st and 2nd graph in Figure~\ref{example_weighted_vertex}.

If $\Gamma'$ consists of three vertices with a loop, then by the definition of weighted signed graph one of the cases of such $\Gamma'$ is as Figure~\ref{three_vertices2}.
\begin{figure}[H]
\begin{tikzpicture}
\begin{scope}[xscale=0.4, yscale=0.4]

\draw[thick](-2,2.7) circle [x radius=0.7cm, y radius=0.7cm];
\fill(-2,-2)circle (5pt);
\fill(-2,2)circle (5pt);
\fill(2,-2)circle (5pt);
\draw[thick](-2,-2)--(-2,2);
\node[above] at (-2,3.4) {$w_{V}(i)$};
\draw[thick](2,-2)--(-2,2);
\node[left] at (-2,-2) {$w_{V}(j)$};
\node[right] at (2,-2) {$w_{V}(k)$};

\end{scope}
\end{tikzpicture}
\caption{Forbidden subgraph on $3$ vertices with a loop for $3$-independence}
\label{three_vertices2} 
\end{figure}
However, we have that 
\begin{align*}
\{w_{V}(i), w_{E}(ij), w_{E}(ik)\}=
\{w_{V}(i), w_{V}(i)-w_{V}(j), w_{V}(i)-w_{V}(k)\} 
\end{align*}
is linearly independent. 
Therefore, this is not $2$-independent.
Similarly, it is easy to verify that the other cases are not $2$-independent. 
We can also check that the graph consisting of three loops is not  $2$-independent.
Therefore, we complete the case when $\Gamma'$ has three vertices.

If $\Gamma'$ consists of two vertices, then by the definition of weighted signed graph  there are the following possibilities for $i<j$ as in Figure~\ref{two_vertices}.
\begin{figure}[h]
\begin{tikzpicture}
\begin{scope}[xscale=0.4, yscale=0.4]

\fill(2,2)circle (5pt);
\fill(2,-2)circle (5pt);
\draw[thick](2,-2)to[out=75,in=-75](2,2);
\draw[dashed](2,-2)to[out=105,in=-105](2,2);
\draw[thick](2,2.7) circle [x radius=0.7cm, y radius=0.7cm];
\node[above] at (2,3.4) {$w_{V}(i)$};
\node[below] at (2,-2) {$w_{V}(j)$};

\fill(6,2)circle (5pt);
\fill(6,-2)circle (5pt);
\draw[dashed](6,-2)to[out=105,in=-105](6,2);
\draw[thick](6,2.7) circle [x radius=0.7cm, y radius=0.7cm];
\node[above] at (6,3.4) {$w_{V}(i)$};
\draw[thick](6,-2.7) circle [x radius=0.7cm, y radius=0.7cm];
\node[below] at (6,-3.4) {$w_{V}(j)$};

\fill(10,2)circle (5pt);
\fill(10,-2)circle (5pt);
\draw[thick](10,-2)to[out=75,in=-75](10,2);
\draw[thick](10,2.7) circle [x radius=0.7cm, y radius=0.7cm];
\node[above] at (10,3.4) {$w_{V}(i)$};
\draw[thick](10,-2.7) circle [x radius=0.7cm, y radius=0.7cm];
\node[below] at (10,-3.4) {$w_{V}(j)$};

\end{scope}
\end{tikzpicture}
\caption{Forbidden subgraphs on $2$ vertices for $3$-independence}
\label{two_vertices}
\end{figure}

From the left, the weights on the loops and edges consist of the following set:
\begin{align*}
& \{e_{i}, e_{i}+e_{j}, e_{i}-e_{j}\}\ {\rm or}\ \{2e_{i}, e_{i}+e_{j}, e_{i}-e_{j}\}; \\
& \{e_{i}, e_{i}+e_{j}, e_{j}\}\ {\rm or}\ \{2e_{i}, e_{i}+e_{j}, 2e_{j}\}; \\
& \{e_{i}, e_{i}-e_{j}, e_{j}\}\ {\rm or}\ \{2e_{i}, e_{i}-e_{j}, 2e_{j}\}.
\end{align*}
All of them are $2$-independent.
Consequently, if $\Gamma'$ consists of two vertices, then it must be the last three graphs in Figure~\ref{example_weighted_vertex}.
\end{proof}

Starting from the next section, we compute $k(G/H)$ case by case for $(G,H)$ in Theorem~\ref{thm:cart}.

\section{Type $A$}
\label{sect:5}
We first compute $k(G/H)$ for the case when $G$ is of type $A_{n}$ for $n\ge 2$ (Note that if $n=2$, then $k(G/H)=2$).
In this case, the corresponding weighted signed graph of $G$ is $K_{n+1}$.

\subsection{Case 1: $H=A_{n-1}\times T^{1}$} 
\label{sect:5.1}
In this case, the independence of $\Gamma(G,H)$ can be computed by removing the edges of the weighted signed graph $K_{n}$ from $K_{n+1}$.
The obtained graph consists of $n$ edges from one-vertex (see Figure~\ref{example_typeA_Case1} for the case when $n=5$).
Therefore, the set of vectors $\Delta_{G,H}/\{\pm 1\}$ is $n$-independent, i.e., linearly independent.
\begin{figure}[H]
\begin{tikzpicture}
\begin{scope}[xscale=0.3, yscale=0.3]

\coordinate (A) at (4,0);
\coordinate (B) at ({4*cos(2*pi/6 r)},{4*sin(2*pi/6 r)});
\coordinate (C) at ({4*cos(4*pi/6 r)},{4*sin(4*pi/6 r)});
\coordinate (D) at ({4*cos(6*pi/6 r)},{4*sin(6*pi/6 r)});
\coordinate (E) at ({4*cos(8*pi/6 r)},{4*sin(8*pi/6 r)});
\coordinate (F) at ({4*cos(10*pi/6 r)},{4*sin(10*pi/6 r)});

\fill(A) circle (5pt);
\node[right] at (A) {$e_{1}$};
\fill(B)circle (5pt);
\node[above] at (B) {$e_{2}$};
\fill(C)circle (5pt);
\node[above] at (C) {$e_{3}$};
\fill(D)circle (5pt);
\node[left] at (D) {$e_{4}$};
\fill(E)circle (5pt);
\node[below] at (E) {$e_{5}$};
\fill(F)circle (5pt);
\node[below] at (F) {$e_{6}$};

\draw (A)--(B)--(C)--(D)--(E)--(F)--(A);
\draw (A)--(C);
\draw (A)--(D);
\draw (A)--(E);
\draw (A)--(F);
\draw (B)--(D);
\draw (B)--(E);
\draw (B)--(F);
\draw (C)--(E);
\draw (C)--(F);
\draw (D)--(F);

\node[right] at (6,0) {$-$};

\coordinate (a) at (18,0);
\coordinate (b) at ({4*cos(2*pi/6 r)+14},{4*sin(2*pi/6 r)});
\coordinate (c) at ({4*cos(4*pi/6 r)+14},{4*sin(4*pi/6 r)});
\coordinate (d) at ({4*cos(6*pi/6 r)+14},{4*sin(6*pi/6 r)});
\coordinate (e) at ({4*cos(8*pi/6 r)+14},{4*sin(8*pi/6 r)});
\coordinate (f) at ({4*cos(10*pi/6 r)+14},{4*sin(10*pi/6 r)});

\fill(a) circle (5pt);
\node[right] at (a) {$e_{1}$};
\fill(b)circle (5pt);
\node[above] at (b) {$e_{2}$};
\fill(c)circle (5pt);
\node[above] at (c) {$e_{3}$};
\fill(d)circle (5pt);
\node[left] at (d) {$e_{4}$};
\fill(e)circle (5pt);
\node[below] at (e) {$e_{5}$};
\fill(f)circle (5pt);
\node[below] at (f) {$e_{6}$};

\draw (a)--(b)--(c)--(d)--(e)--(a);
\draw (a)--(c)--(e)--(b)--(d)--(a);

\node[right] at (20,0) {$=$};

\coordinate (1) at (32,0);
\coordinate (2) at ({4*cos(2*pi/6 r)+28},{4*sin(2*pi/6 r)});
\coordinate (3) at ({4*cos(4*pi/6 r)+28},{4*sin(4*pi/6 r)});
\coordinate (4) at ({4*cos(6*pi/6 r)+28},{4*sin(6*pi/6 r)});
\coordinate (5) at ({4*cos(8*pi/6 r)+28},{4*sin(8*pi/6 r)});
\coordinate (6) at ({4*cos(10*pi/6 r)+28},{4*sin(10*pi/6 r)});

\fill(1) circle (5pt);
\node[right] at (1) {$e_{1}$};
\fill(2)circle (5pt);
\node[above] at (2) {$e_{2}$};
\fill(3)circle (5pt);
\node[above] at (3) {$e_{3}$};
\fill(4)circle (5pt);
\node[left] at (4) {$e_{4}$};
\fill(5)circle (5pt);
\node[below] at (5) {$e_{5}$};
\fill(6)circle (5pt);
\node[below] at (6) {$e_{6}$};

\draw (6)--(1);
\draw (6)--(2);
\draw (6)--(3);
\draw (6)--(4);
\draw (6)--(5);

\end{scope}
\end{tikzpicture}
\caption{The weighted signed graph corresponding to $A_{5}$ and $A_{4}\times T^1$. The obtained graph consists of $5$ out-going edges from the vertex $e_{6}$. This means that the obtained vectors are $\{e_{6}-e_{i}\ |\ i=1,\ldots, 5\}$; therefore, this is $5$-independent (linearly independent)}
\label{example_typeA_Case1}
\end{figure}
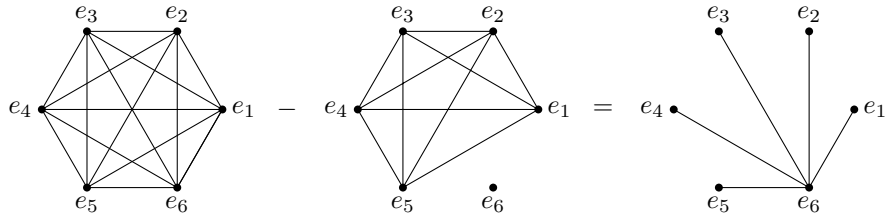

This shows Theorem~\ref{main} (1) Type $A_{n}$.

\subsection{Case 2: $H=A_{i}\times A_{n-1-i}\times T^{1}$ $(0<i<n-1)$} 
\label{sect:5.2}
This case can be computed by removing the edges of the weighted signed graph $K_{i+1}\sqcup K_{n-i}$ from $K_{n+1}$.
The obtained graph is the bipartite graph $K_{i+1,n-i}$.
Then, it does not have a $3$-cycle but it has the $4$-cycle which appears in the 1st left Figure~\ref{example_weighted_vertex2} (see Figure~\ref{example_typeA_Case2} for the case when $n=5$ and $i=3$).
Therefore, the set of vectors $\Delta_{G,H}/\{\pm 1\}$ is $3$-independent.
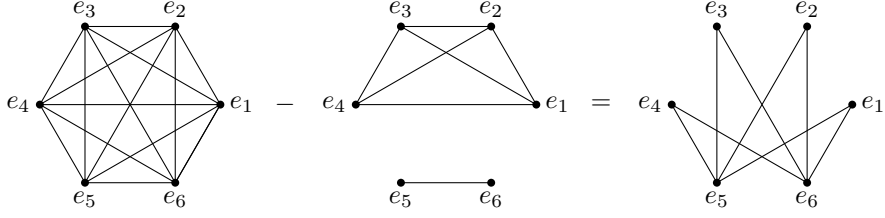
\begin{figure}[H]
\begin{tikzpicture}
\begin{scope}[xscale=0.3, yscale=0.3]
\coordinate (A) at (4,0);
\coordinate (B) at ({4*cos(2*pi/6 r)},{4*sin(2*pi/6 r)});
\coordinate (C) at ({4*cos(4*pi/6 r)},{4*sin(4*pi/6 r)});
\coordinate (D) at ({4*cos(6*pi/6 r)},{4*sin(6*pi/6 r)});
\coordinate (E) at ({4*cos(8*pi/6 r)},{4*sin(8*pi/6 r)});
\coordinate (F) at ({4*cos(10*pi/6 r)},{4*sin(10*pi/6 r)});

\fill(A) circle (5pt);
\node[right] at (A) {$e_{1}$};
\fill(B)circle (5pt);
\node[above] at (B) {$e_{2}$};
\fill(C)circle (5pt);
\node[above] at (C) {$e_{3}$};
\fill(D)circle (5pt);
\node[left] at (D) {$e_{4}$};
\fill(E)circle (5pt);
\node[below] at (E) {$e_{5}$};
\fill(F)circle (5pt);
\node[below] at (F) {$e_{6}$};

\draw (A)--(B)--(C)--(D)--(E)--(F)--(A);
\draw (A)--(C);
\draw (A)--(D);
\draw (A)--(E);
\draw (A)--(F);
\draw (B)--(D);
\draw (B)--(E);
\draw (B)--(F);
\draw (C)--(E);
\draw (C)--(F);
\draw (D)--(F);

\node[right] at (6,0) {$-$};

\coordinate (a) at (18,0);
\coordinate (b) at ({4*cos(2*pi/6 r)+14},{4*sin(2*pi/6 r)});
\coordinate (c) at ({4*cos(4*pi/6 r)+14},{4*sin(4*pi/6 r)});
\coordinate (d) at ({4*cos(6*pi/6 r)+14},{4*sin(6*pi/6 r)});
\coordinate (e) at ({4*cos(8*pi/6 r)+14},{4*sin(8*pi/6 r)});
\coordinate (f) at ({4*cos(10*pi/6 r)+14},{4*sin(10*pi/6 r)});

\fill(a) circle (5pt);
\node[right] at (a) {$e_{1}$};
\fill(b)circle (5pt);
\node[above] at (b) {$e_{2}$};
\fill(c)circle (5pt);
\node[above] at (c) {$e_{3}$};
\fill(d)circle (5pt);
\node[left] at (d) {$e_{4}$};
\fill(e)circle (5pt);
\node[below] at (e) {$e_{5}$};
\fill(f)circle (5pt);
\node[below] at (f) {$e_{6}$};

\draw (a)--(b)--(c)--(d)--(a);
\draw (a)--(c);
\draw (b)--(d);

\draw (e)--(f);

\node[right] at (20,0) {$=$};

\coordinate (1) at (32,0);
\coordinate (2) at ({4*cos(2*pi/6 r)+28},{4*sin(2*pi/6 r)});
\coordinate (3) at ({4*cos(4*pi/6 r)+28},{4*sin(4*pi/6 r)});
\coordinate (4) at ({4*cos(6*pi/6 r)+28},{4*sin(6*pi/6 r)});
\coordinate (5) at ({4*cos(8*pi/6 r)+28},{4*sin(8*pi/6 r)});
\coordinate (6) at ({4*cos(10*pi/6 r)+28},{4*sin(10*pi/6 r)});

\fill(1) circle (5pt);
\node[right] at (1) {$e_{1}$};
\fill(2)circle (5pt);
\node[above] at (2) {$e_{2}$};
\fill(3)circle (5pt);
\node[above] at (3) {$e_{3}$};
\fill(4)circle (5pt);
\node[left] at (4) {$e_{4}$};
\fill(5)circle (5pt);
\node[below] at (5) {$e_{5}$};
\fill(6)circle (5pt);
\node[below] at (6) {$e_{6}$};

\draw (1)--(6);
\draw (1)--(5);
\draw (2)--(6);
\draw (2)--(5);
\draw (3)--(6);
\draw (3)--(5);
\draw (4)--(6);
\draw (4)--(5);

\end{scope}
\end{tikzpicture}
\caption{The weighted signed graph corresponding to $A_{5}$ and $A_{3}\times A_{1}\times T^1$.
The obtained graph is the bipartite ($2$-partite) graph $K_{4,2}$.
The bipartite graph does not have $3$-cycle but have $4$-cycle.
This means that there are $4$ vertices whose weights are $3$-independent; therefore, this is $3$-independent}
\label{example_typeA_Case2}
\end{figure}

This shows Theorem~\ref{main} (2) $(A_{n}, A_{i}\times A_{n-i-1}\times T^{1})$ for $0<i<n-1$.

This establishes the computations of $k(G/H)$ if $G$ is of type $A_{n}$.

\section{Type $B$ and $C$}
\label{sect:6}

In this section, we compute $k(G/H)$ for the case when $G$ is of type $B_{n}$ for $n\ge 2$ and of type $C_{n}$ for $n\ge 2$.
In this case, the corresponding weighted signed graph of $G$ is $K_{n}^{\circ}$.

\subsection{Case 1: $(G,H)=(B_{n},D_{i}\times B_{n-i})$ $(1\le i< n)$}
\label{sect:6.1}
This case can be computed by removing the edges of the weighted signed graph $\pm K_{i}\sqcup K_{n-i}^{\circ}$ from $K_{n}^{\circ}$ with $w_{V}(i)=e_{i}$, where we regard $\pm K_{1}=\emptyset$.
The obtained graph has a subgraph of one loop with multiple edge (see Figure~\ref{example_typeB_Case1} for $(n,i)=(4,1)$ and Figure~\ref{example_typeB_Case2} for $(n,i)=(6,4)$).
This is $2$-independent because there is a subgraph in  Figure~\ref{example_weighted_vertex}.
Therefore, the set of vectors $\Delta_{G,H}/\{\pm 1\}$ is $2$-independent.

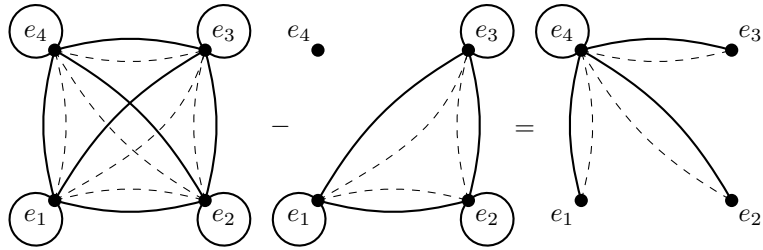
\begin{figure}[H]
\begin{tikzpicture}
\begin{scope}[xscale=0.5, yscale=0.5]

\fill(8,-2)circle (5pt);
\fill(12,-2)circle (5pt);
\fill(12,2)circle (5pt);
\fill(8,2)circle (5pt);

\node[below] at (7.5,-2) {$e_{1}$};
\node[below] at (12.5,-2) {$e_{2}$};
\node[above] at (12.5,2) {$e_{3}$};
\node[above] at (7.5,2) {$e_{4}$};

\draw[dashed](8,-2)to[out=15,in=165](12,-2);
\draw[thick](8,-2)to[out=-15,in=-165](12,-2);
\draw[thick](8,2)to[out=15,in=165](12,2);
\draw[dashed](8,2)to[out=-15,in=-165](12,2);
\draw[thick](12,-2)to[out=75,in=-75](12,2);
\draw[dashed](12,-2)to[out=105,in=-105](12,2);
\draw[thick](8,-2)to[out=105,in=-105](8,2);
\draw[dashed](8,-2)to[out=75,in=-75](8,2);
\draw[thick](8,-2)to[out=60,in=-150](12,2);
\draw[dashed](8,-2)to[out=30,in=-105](12,2);
\draw[thick](8,2)to[out=-30,in=120](12,-2);
\draw[dashed](8,2)to[out=-60,in=150](12,-2);

\draw[thick](7.5,-2.5) circle [x radius=0.7cm, y radius=0.7cm, rotate=30];
\draw[thick](7.5,2.5) circle [x radius=0.7cm, y radius=0.7cm, rotate=30];
\draw[thick](12.5,-2.5) circle [x radius=0.7cm, y radius=0.7cm, rotate=30];
\draw[thick](12.5,2.5) circle [x radius=0.7cm, y radius=0.7cm, rotate=30];

\node[above] at (14,-0.5) {$-$};


\fill(15,-2)circle (5pt);
\fill(19,-2)circle (5pt);
\fill(19,2)circle (5pt);
\fill(15,2)circle (5pt);

\node[below] at (14.5,-2) {$e_{1}$};
\node[below] at (19.5,-2) {$e_{2}$};
\node[above] at (19.5,2) {$e_{3}$};
\node[above] at (14.5,2) {$e_{4}$};

\draw[dashed](15,-2)to[out=15,in=165](19,-2);
\draw[thick](15,-2)to[out=-15,in=-165](19,-2);
\draw[thick](19,-2)to[out=75,in=-75](19,2);
\draw[dashed](19,-2)to[out=105,in=-105](19,2);
\draw[thick](15,-2)to[out=60,in=-150](19,2);
\draw[dashed](15,-2)to[out=30,in=-105](19,2);

\draw[thick](14.5,-2.5) circle [x radius=0.7cm, y radius=0.7cm, rotate=30];
\draw[thick](19.5,-2.5) circle [x radius=0.7cm, y radius=0.7cm, rotate=30];
\draw[thick](19.5,2.5) circle [x radius=0.7cm, y radius=0.7cm, rotate=30];

\node[above] at (20.5,-0.5) {$=$};


\fill(22,-2)circle (5pt);
\fill(26,-2)circle (5pt);
\fill(26,2)circle (5pt);
\fill(22,2)circle (5pt);

\node[below] at (21.5,-2) {$e_{1}$};
\node[below] at (26.5,-2) {$e_{2}$};
\node[above] at (26.5,2) {$e_{3}$};
\node[above] at (21.5,2) {$e_{4}$};

\draw[thick](22,2)to[out=15,in=165](26,2);
\draw[dashed](22,2)to[out=-15,in=-165](26,2);
\draw[thick](22,-2)to[out=105,in=-105](22,2);
\draw[dashed](22,-2)to[out=75,in=-75](22,2);
\draw[thick](22,2)to[out=-30,in=120](26,-2);
\draw[dashed](22,2)to[out=-60,in=150](26,-2);

\draw[thick](21.5,2.5) circle [x radius=0.7cm, y radius=0.7cm, rotate=30];

\end{scope}
\end{tikzpicture}
\caption{The weighted signed graph corresponding to $B_{4}$ and $D_{1}\times B_{3}=T^{1}\times B_{3}$. This is $2$-independent}
\label{example_typeB_Case1}
\end{figure}

\begin{figure}[H]
\begin{tikzpicture}
\begin{scope}[xscale=0.3, yscale=0.3]

\coordinate (A) at (5,0);
\coordinate (B) at ({5*cos(2*pi/6 r)},{5*sin(2*pi/6 r)});
\coordinate (C) at ({5*cos(4*pi/6 r)},{5*sin(4*pi/6 r)});
\coordinate (D) at ({5*cos(6*pi/6 r)},{5*sin(6*pi/6 r)});
\coordinate (E) at ({5*cos(8*pi/6 r)},{5*sin(8*pi/6 r)});
\coordinate (F) at ({5*cos(10*pi/6 r)},{5*sin(10*pi/6 r)});

\draw[thick](5.5,0) circle (0.5cm);
\draw[thick](3,4.5) circle [x radius=0.5cm, y radius=0.5cm, rotate=30];
\draw[thick](-3,4.5) circle [x radius=0.5cm, y radius=0.5cm, rotate=30];
\draw[thick](-5.5,0) circle (0.5cm);
\draw[thick](-3,-4.5) circle [x radius=0.5cm, y radius=0.5cm, rotate=30];
\draw[thick](3,-4.5) circle [x radius=0.5cm, y radius=0.5cm, rotate=30];

\fill(A) circle (5pt);
\node[label={300:$e_{1}$}] at (A) {};
\fill(B)circle (5pt);
\node[label={100:$e_{2}$}] at (B) {};
\fill(C)circle (5pt);
\node[label={80:$e_{3}$}] at (C) {};
\fill(D)circle (5pt);
\node[label={100:$e_{4}$}] at (D) {};
\fill(E)circle (5pt);
\node[label={300:$e_{5}$}] at (E) {};
\fill(F)circle (5pt);
\node[label={260:$e_{6}$}] at (F) {};

\draw[thick] (A)--(B)--(C)--(D)--(E)--(F)--(A);
\draw[thick] (A)--(C);
\draw[thick] (A)--(D);
\draw[thick] (A)--(E);
\draw[thick] (B)--(D);
\draw[thick] (B)--(E);
\draw[thick] (B)--(F);
\draw[thick] (C)--(D);
\draw[thick] (C)--(E);
\draw[thick] (C)--(F);
\draw[thick] (D)--(F);

\draw[dashed](A)to[out=105,in=-45](B); 
\draw[dashed](A)to[out=135,in=-15](C); 
\draw[dashed](A)to[out=165,in=15](D); 
\draw[dashed](A)to[out=195,in=45](E);
\draw[dashed](A)to[out=225,in=75](F);
\draw[dashed](B)to[out=165,in=15](C);
\draw[dashed](B)to[out=195,in=45](D);
\draw[dashed](B)to[out=225,in=75](E);
\draw[dashed](B)to[out=255,in=105](F);
\draw[dashed](C)to[out=225,in=75](D);
\draw[dashed](C)to[out=255,in=105](E);
\draw[dashed](C)to[out=285,in=135](F);
\draw[dashed](D)to[out=285,in=135](E);
\draw[dashed](D)to[out=315,in=165](F);
\draw[dashed](F)to[out=165,in=15](E); 

\node[right] at (7,0) {$-$};

\coordinate (a) at (20,0);
\coordinate (b) at ({5*cos(2*pi/6 r)+15},{5*sin(2*pi/6 r)});
\coordinate (c) at ({5*cos(4*pi/6 r)+15},{5*sin(4*pi/6 r)});
\coordinate (d) at ({5*cos(6*pi/6 r)+15},{5*sin(6*pi/6 r)});
\coordinate (e) at ({5*cos(8*pi/6 r)+15},{5*sin(8*pi/6 r)});
\coordinate (f) at ({5*cos(10*pi/6 r)+15},{5*sin(10*pi/6 r)});

\draw[thick](12,-4.5) circle [x radius=0.5cm, y radius=0.5cm, rotate=30];
\draw[thick](18,-4.5) circle [x radius=0.5cm, y radius=0.5cm, rotate=30];

\fill(a) circle (5pt);
\node[right] at (a) {$e_{1}$};
\fill(b)circle (5pt);
\node[above] at (b) {$e_{2}$};
\fill(c)circle (5pt);
\node[above] at (c) {$e_{3}$};
\fill(d)circle (5pt);
\node[left] at (d) {$e_{4}$};
\fill(e)circle (5pt);
\node[label={300:$e_{5}$}] at (e) {};
\fill(f)circle (5pt);
\node[label={260:$e_{6}$}] at (f) {};

\draw[thick] (a)--(b)--(c)--(d)--(a);
\draw[thick] (a)--(c);
\draw[thick] (b)--(d);
\draw[thick] (e)--(f);

\draw[dashed](a)to[out=105,in=-45](b); 
\draw[dashed](a)to[out=135,in=-15](c); 
\draw[dashed](a)to[out=165,in=15](d); 
\draw[dashed](b)to[out=165,in=15](c);
\draw[dashed](b)to[out=195,in=45](d);
\draw[dashed](c)to[out=225,in=75](d);

\draw[dashed](f)to[out=165,in=15](e);

\node[right] at (22,0) {$=$};

\coordinate (1) at (36,0);
\coordinate (2) at ({5*cos(2*pi/6 r)+31},{5*sin(2*pi/6 r)});
\coordinate (3) at ({5*cos(4*pi/6 r)+31},{5*sin(4*pi/6 r)});
\coordinate (4) at ({5*cos(6*pi/6 r)+31},{5*sin(6*pi/6 r)});
\coordinate (5) at ({5*cos(8*pi/6 r)+31},{5*sin(8*pi/6 r)});
\coordinate (6) at ({5*cos(10*pi/6 r)+31},{5*sin(10*pi/6 r)});

\draw[thick](36.5,0) circle (0.5cm);
\draw[thick](34,4.5) circle [x radius=0.5cm, y radius=0.5cm, rotate=30];
\draw[thick](28,4.5) circle [x radius=0.5cm, y radius=0.5cm, rotate=30];
\draw[thick](25.5,0) circle (0.5cm);

\fill(1) circle (5pt);
\node[label={300:$e_{1}$}] at (1) {};
\fill(2)circle (5pt);
\node[label={100:$e_{2}$}] at (2) {};
\fill(3)circle (5pt);
\node[label={80:$e_{3}$}] at (3) {};
\fill(4)circle (5pt);
\node[label={100:$e_{4}$}] at (4) {};
\fill(5)circle (5pt);
\node[label={300:$e_{5}$}] at (5) {};
\fill(6)circle (5pt);
\node[label={260:$e_{6}$}] at (6) {};

\draw[thick] (5)--(1);
\draw[thick] (5)--(2);
\draw[thick] (5)--(3);
\draw[thick] (5)--(4);
\draw[thick] (6)--(1);
\draw[thick] (6)--(2);
\draw[thick] (6)--(3);
\draw[thick] (6)--(4);

\draw[dashed](1)to[out=195,in=45](5);
\draw[dashed](1)to[out=225,in=75](6);
\draw[dashed](2)to[out=225,in=75](5);
\draw[dashed](2)to[out=255,in=105](6);
\draw[dashed](3)to[out=255,in=105](5);
\draw[dashed](3)to[out=285,in=135](6);
\draw[dashed](4)to[out=285,in=135](5);
\draw[dashed](4)to[out=315,in=165](6);

\end{scope}
\end{tikzpicture}
\caption{The weighted signed graph corresponding to $B_{6}$ and $D_{4}\times B_{2}$. 
This is $2$-independent}
\label{example_typeB_Case2}
\end{figure}
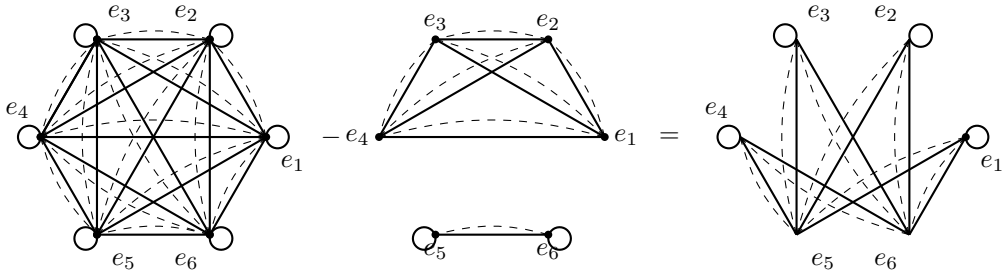

This proves that $(B_{n},D_{i}\times B_{n-i})$ $(1\le i< n)$ is one of the cases in Theorem~\ref{main} (3).

\subsection{Case 2: $(G,H)=(B_{n},D_{n})$}
\label{sect:6.2}

This case can be computed by removing the edges of the weighted signed graph $\pm K_{n}$ from $K_{n}^{\circ}$ with $w_{V}(i)=e_{i}$.
The obtained graph consists of the $n$ loops (see Figure~\ref{example_typeB_Case3} for the case when $n=4$).
Therefore, the set of vectors $\Delta_{G,H}/\{\pm 1\}$ is $n$-independent.

\begin{figure}[H]
\begin{tikzpicture}
\begin{scope}[xscale=0.5, yscale=0.5]

\fill(8,-2)circle (5pt);
\fill(12,-2)circle (5pt);
\fill(12,2)circle (5pt);
\fill(8,2)circle (5pt);

\node[below] at (7.5,-2) {$e_{1}$};
\node[below] at (12.5,-2) {$e_{2}$};
\node[above] at (12.5,2) {$e_{3}$};
\node[above] at (7.5,2) {$e_{4}$};

\draw[dashed](8,-2)to[out=15,in=165](12,-2);
\draw[thick](8,-2)to[out=-15,in=-165](12,-2);
\draw[thick](8,2)to[out=15,in=165](12,2);
\draw[dashed](8,2)to[out=-15,in=-165](12,2);
\draw[thick](12,-2)to[out=75,in=-75](12,2);
\draw[dashed](12,-2)to[out=105,in=-105](12,2);
\draw[thick](8,-2)to[out=105,in=-105](8,2);
\draw[dashed](8,-2)to[out=75,in=-75](8,2);
\draw[thick](8,-2)to[out=60,in=-150](12,2);
\draw[dashed](8,-2)to[out=30,in=-105](12,2);
\draw[thick](8,2)to[out=-30,in=120](12,-2);
\draw[dashed](8,2)to[out=-60,in=150](12,-2);

\draw[thick](7.5,-2.5) circle [x radius=0.7cm, y radius=0.7cm, rotate=30];
\draw[thick](7.5,2.5) circle [x radius=0.7cm, y radius=0.7cm, rotate=30];
\draw[thick](12.5,-2.5) circle [x radius=0.7cm, y radius=0.7cm, rotate=30];
\draw[thick](12.5,2.5) circle [x radius=0.7cm, y radius=0.7cm, rotate=30];

\node[above] at (14,-0.5) {$-$};


\fill(15,-2)circle (5pt);
\fill(19,-2)circle (5pt);
\fill(19,2)circle (5pt);
\fill(15,2)circle (5pt);

\node[below] at (14.5,-2) {$e_{1}$};
\node[below] at (19.5,-2) {$e_{2}$};
\node[above] at (19.5,2) {$e_{3}$};
\node[above] at (14.5,2) {$e_{4}$};

\draw[dashed](15,-2)to[out=15,in=165](19,-2);
\draw[thick](15,-2)to[out=-15,in=-165](19,-2);
\draw[thick](15,2)to[out=15,in=165](19,2);
\draw[dashed](15,2)to[out=-15,in=-165](19,2);
\draw[thick](19,-2)to[out=75,in=-75](19,2);
\draw[dashed](19,-2)to[out=105,in=-105](19,2);
\draw[thick](15,-2)to[out=105,in=-105](15,2);
\draw[dashed](15,-2)to[out=75,in=-75](15,2);
\draw[thick](15,-2)to[out=60,in=-150](19,2);
\draw[dashed](15,-2)to[out=30,in=-105](19,2);
\draw[thick](15,2)to[out=-30,in=120](19,-2);
\draw[dashed](15,2)to[out=-60,in=150](19,-2);

\node[above] at (20.5,-0.5) {$=$};


\fill(22,-2)circle (5pt);
\fill(26,-2)circle (5pt);
\fill(26,2)circle (5pt);
\fill(22,2)circle (5pt);

\node[below] at (21.5,-2) {$e_{1}$};
\node[below] at (26.5,-2) {$e_{2}$};
\node[above] at (26.5,2) {$e_{3}$};
\node[above] at (21.5,2) {$e_{4}$};

\draw[thick](21.5,-2.5) circle [x radius=0.7cm, y radius=0.7cm, rotate=30];
\draw[thick](21.5,2.5) circle [x radius=0.7cm, y radius=0.7cm, rotate=30];
\draw[thick](26.5,-2.5) circle [x radius=0.7cm, y radius=0.7cm, rotate=30];
\draw[thick](26.5,2.5) circle [x radius=0.7cm, y radius=0.7cm, rotate=30];

\end{scope}
\end{tikzpicture}
\caption{The weighted signed graph corresponding to $B_{4}$ and $D_{4}$. This is $n$-independent}
\label{example_typeB_Case3}
\end{figure}
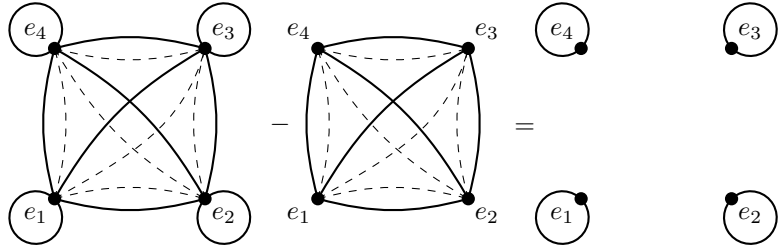

This shows Theorem~\ref{main} (1) Type $B_{n}$.

\subsection{Case 3: $(G,H)=(C_{n},A_{n-1}\times T^{1})$}
\label{sect:6.3}
 
This case can be computed by removing the edges of the weighted signed graph $K_{n}$ from $K_{n}^{\circ}$ with $w_{V}(i)=2e_{i}$.
The obtained graph has a subgraph of two loops connected by a dotted edge (see Figure~\ref{example_typeC_Case1} for the case when $n=4$).
This is $2$-independent because there is a subgraph in  Figure~\ref{example_weighted_vertex}.
Therefore, the set of vectors $\Delta_{G,H}/\{\pm 1\}$ is $2$-independent.

\begin{figure}[H]
\begin{tikzpicture}
\begin{scope}[xscale=0.5, yscale=0.5]

\fill(8,-2)circle (5pt);
\fill(12,-2)circle (5pt);
\fill(12,2)circle (5pt);
\fill(8,2)circle (5pt);

\node[below] at (7.5,-2) {$2e_{1}$};
\node[below] at (12.5,-2) {$2e_{2}$};
\node[above] at (12.5,2) {$2e_{3}$};
\node[above] at (7.5,2) {$2e_{4}$};

\draw[dashed](8,-2)to[out=15,in=165](12,-2);
\draw[thick](8,-2)to[out=-15,in=-165](12,-2);
\draw[thick](8,2)to[out=15,in=165](12,2);
\draw[dashed](8,2)to[out=-15,in=-165](12,2);
\draw[thick](12,-2)to[out=75,in=-75](12,2);
\draw[dashed](12,-2)to[out=105,in=-105](12,2);
\draw[thick](8,-2)to[out=105,in=-105](8,2);
\draw[dashed](8,-2)to[out=75,in=-75](8,2);
\draw[thick](8,-2)to[out=60,in=-150](12,2);
\draw[dashed](8,-2)to[out=30,in=-105](12,2);
\draw[thick](8,2)to[out=-30,in=120](12,-2);
\draw[dashed](8,2)to[out=-60,in=150](12,-2);

\draw[thick](7.5,-2.5) circle [x radius=0.7cm, y radius=0.7cm, rotate=30];
\draw[thick](7.5,2.5) circle [x radius=0.7cm, y radius=0.7cm, rotate=30];
\draw[thick](12.5,-2.5) circle [x radius=0.7cm, y radius=0.7cm, rotate=30];
\draw[thick](12.5,2.5) circle [x radius=0.7cm, y radius=0.7cm, rotate=30];

\node[above] at (14,-0.5) {$-$};


\fill(15,-2)circle (5pt);
\fill(19,-2)circle (5pt);
\fill(19,2)circle (5pt);
\fill(15,2)circle (5pt);

\node[below] at (14.5,-2) {$e_{1}$};
\node[below] at (19.5,-2) {$e_{2}$};
\node[above] at (19.5,2) {$e_{3}$};
\node[above] at (14.5,2) {$e_{4}$};

\draw[thick](15,-2)to[out=-15,in=-165](19,-2);
\draw[thick](15,2)to[out=15,in=165](19,2);
\draw[thick](19,-2)to[out=75,in=-75](19,2);
\draw[thick](15,-2)to[out=105,in=-105](15,2);
\draw[thick](15,-2)to[out=60,in=-150](19,2);
\draw[thick](15,2)to[out=-30,in=120](19,-2);

\node[above] at (20.5,-0.5) {$=$};


\fill(22,-2)circle (5pt);
\fill(26,-2)circle (5pt);
\fill(26,2)circle (5pt);
\fill(22,2)circle (5pt);

\node[below] at (21.5,-2) {$2e_{1}$};
\node[below] at (26.5,-2) {$2e_{2}$};
\node[above] at (26.5,2) {$2e_{3}$};
\node[above] at (21.5,2) {$2e_{4}$};

\draw[dashed](22,-2)to[out=15,in=165](26,-2);
\draw[dashed](22,2)to[out=-15,in=-165](26,2);
\draw[dashed](26,-2)to[out=105,in=-105](26,2);
\draw[dashed](22,-2)to[out=75,in=-75](22,2);
\draw[dashed](22,-2)to[out=30,in=-105](26,2);
\draw[dashed](22,2)to[out=-60,in=150](26,-2);

\draw[thick](21.5,-2.5) circle [x radius=0.7cm, y radius=0.7cm, rotate=30];
\draw[thick](21.5,2.5) circle [x radius=0.7cm, y radius=0.7cm, rotate=30];
\draw[thick](26.5,-2.5) circle [x radius=0.7cm, y radius=0.7cm, rotate=30];
\draw[thick](26.5,2.5) circle [x radius=0.7cm, y radius=0.7cm, rotate=30];

\end{scope}
\end{tikzpicture}
\caption{The weighted signed graph corresponding to $C_{4}$ and $A_{3}\times T^{1}$. This is $2$-independent}
\label{example_typeC_Case1}
\end{figure}
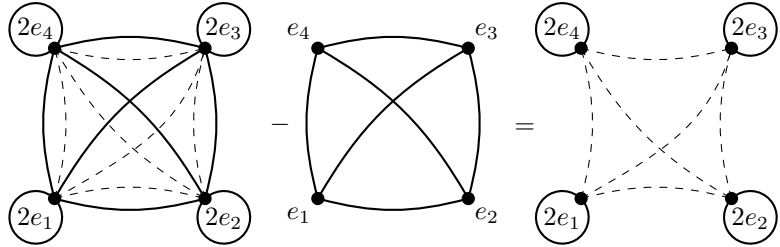

This proves that $(C_{n},A_{n-1}\times T^{1})$ is one of the cases in Theorem~\ref{main} (3).

\subsection{Case 4: $(G,H)=(C_{n},C_{i}\times C_{n-i})$ $(1\le i<n)$}
\label{sect:6.4}
For $n=2$ one has $(G,H)=(C_{2},C_{1}\times C_{1})$ and $\Delta_{G,H}/\{\pm 1\}=\lb e_{1}+e_{2},e_{1}-e_{2}\rb$ is $2$-independent.
Let $n\geq 3$.
This case can be computed by removing the edges of the weighted signed graph $K_{i}^{\circ}\sqcup K_{n-i}^{\circ}$ from $K_{n}^{\circ}$.
The obtained graph combines the bipartite graph and the dashed bipartite graph.
Then, it does not have a $3$-cycle but it has the $4$-cycle which appears in Figure~\ref{example_weighted_vertex2} (see Figure~\ref{example_typeC_Case2} for the case when $n=6$ and $i=4$).
Therefore, the set of vectors $\Delta_{G,H}/\{\pm 1\}$ is $3$-independent.

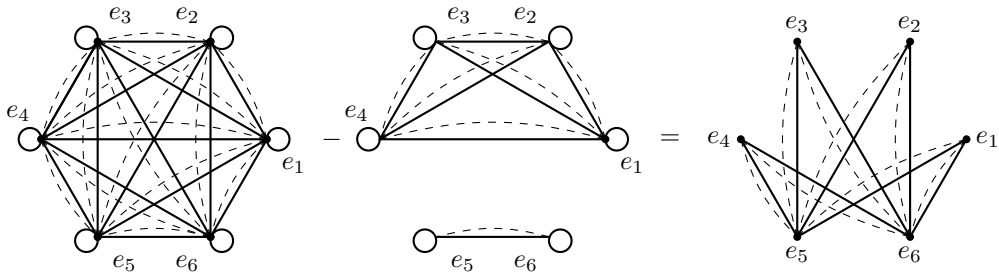
\begin{figure}[H]
\begin{tikzpicture}
\begin{scope}[xscale=0.3, yscale=0.3]

\coordinate (A) at (5,0);
\coordinate (B) at ({5*cos(2*pi/6 r)},{5*sin(2*pi/6 r)});
\coordinate (C) at ({5*cos(4*pi/6 r)},{5*sin(4*pi/6 r)});
\coordinate (D) at ({5*cos(6*pi/6 r)},{5*sin(6*pi/6 r)});
\coordinate (E) at ({5*cos(8*pi/6 r)},{5*sin(8*pi/6 r)});
\coordinate (F) at ({5*cos(10*pi/6 r)},{5*sin(10*pi/6 r)});

\draw[thick](5.5,0) circle (0.5cm);
\draw[thick](3,4.5) circle [x radius=0.5cm, y radius=0.5cm, rotate=30];
\draw[thick](-3,4.5) circle [x radius=0.5cm, y radius=0.5cm, rotate=30];
\draw[thick](-5.5,0) circle (0.5cm);
\draw[thick](-3,-4.5) circle [x radius=0.5cm, y radius=0.5cm, rotate=30];
\draw[thick](3,-4.5) circle [x radius=0.5cm, y radius=0.5cm, rotate=30];

\fill(A) circle (5pt);

\node[label={300:$e_{1}$}] at (A) {};
\fill(B)circle (5pt);
\node[label={100:$e_{2}$}] at (B) {};
\fill(C)circle (5pt);
\node[label={80:$e_{3}$}] at (C) {};
\fill(D)circle (5pt);
\node[label={100:$e_{4}$}] at (D) {};
\fill(E)circle (5pt);
\node[label={300:$e_{5}$}] at (E) {};
\fill(F)circle (5pt);
\node[label={260:$e_{6}$}] at (F) {};

\draw[thick] (A)--(B)--(C)--(D)--(E)--(F)--(A);
\draw[thick] (A)--(C);
\draw[thick] (A)--(D);
\draw[thick] (A)--(E);
\draw[thick] (B)--(D);
\draw[thick] (B)--(E);
\draw[thick] (B)--(F);
\draw[thick] (C)--(D);
\draw[thick] (C)--(E);
\draw[thick] (C)--(F);
\draw[thick] (D)--(F);

\draw[dashed](A)to[out=105,in=-45](B); 
\draw[dashed](A)to[out=135,in=-15](C); 
\draw[dashed](A)to[out=165,in=15](D); 
\draw[dashed](A)to[out=195,in=45](E);
\draw[dashed](A)to[out=225,in=75](F);
\draw[dashed](B)to[out=165,in=15](C);
\draw[dashed](B)to[out=195,in=45](D);
\draw[dashed](B)to[out=225,in=75](E);
\draw[dashed](B)to[out=255,in=105](F);
\draw[dashed](C)to[out=225,in=75](D);
\draw[dashed](C)to[out=255,in=105](E);
\draw[dashed](C)to[out=285,in=135](F);
\draw[dashed](D)to[out=285,in=135](E);
\draw[dashed](D)to[out=315,in=165](F);
\draw[dashed](F)to[out=165,in=15](E); 

\node[right] at (7,0) {$-$};

\coordinate (a) at (20,0);
\coordinate (b) at ({5*cos(2*pi/6 r)+15},{5*sin(2*pi/6 r)});
\coordinate (c) at ({5*cos(4*pi/6 r)+15},{5*sin(4*pi/6 r)});
\coordinate (d) at ({5*cos(6*pi/6 r)+15},{5*sin(6*pi/6 r)});
\coordinate (e) at ({5*cos(8*pi/6 r)+15},{5*sin(8*pi/6 r)});
\coordinate (f) at ({5*cos(10*pi/6 r)+15},{5*sin(10*pi/6 r)});

\draw[thick](20.5,0) circle (0.5cm);
\draw[thick](12,4.5) circle [x radius=0.5cm, y radius=0.5cm, rotate=30];
\draw[thick](18,4.5) circle [x radius=0.5cm, y radius=0.5cm, rotate=30];
\draw[thick](9.5,0) circle (0.5cm);
\draw[thick](12,-4.5) circle [x radius=0.5cm, y radius=0.5cm, rotate=30];
\draw[thick](18,-4.5) circle [x radius=0.5cm, y radius=0.5cm, rotate=30];

\fill(a) circle (5pt);
\node[label={300:$e_{1}$}] at (a) {};
\fill(b)circle (5pt);
\node[label={100:$e_{2}$}] at (b) {};
\fill(c)circle (5pt);
\node[label={80:$e_{3}$}] at (c) {};
\fill(d)circle (5pt);
\node[label={100:$e_{4}$}] at (d) {};
\fill(e)circle (5pt);
\node[label={300:$e_{5}$}] at (e) {};
\fill(f)circle (5pt);
\node[label={260:$e_{6}$}] at (f) {};

\draw[thick] (a)--(b)--(c)--(d)--(a);
\draw[thick] (a)--(c);
\draw[thick] (b)--(d);
\draw[thick] (e)--(f);

\draw[dashed](a)to[out=105,in=-45](b); 
\draw[dashed](a)to[out=135,in=-15](c); 
\draw[dashed](a)to[out=165,in=15](d); 
\draw[dashed](b)to[out=165,in=15](c);
\draw[dashed](b)to[out=195,in=45](d);
\draw[dashed](c)to[out=225,in=75](d);

\draw[dashed](f)to[out=165,in=15](e); 

\node[right] at (22,0) {$=$};

\coordinate (1) at (36,0);
\coordinate (2) at ({5*cos(2*pi/6 r)+31},{5*sin(2*pi/6 r)});
\coordinate (3) at ({5*cos(4*pi/6 r)+31},{5*sin(4*pi/6 r)});
\coordinate (4) at ({5*cos(6*pi/6 r)+31},{5*sin(6*pi/6 r)});
\coordinate (5) at ({5*cos(8*pi/6 r)+31},{5*sin(8*pi/6 r)});
\coordinate (6) at ({5*cos(10*pi/6 r)+31},{5*sin(10*pi/6 r)});

\fill(1) circle (5pt);
\node[right] at (1) {$e_{1}$};
\fill(2)circle (5pt);
\node[above] at (2) {$e_{2}$};
\fill(3)circle (5pt);
\node[above] at (3) {$e_{3}$};
\fill(4)circle (5pt);
\node[left] at (4) {$e_{4}$};
\fill(5)circle (5pt);
\node[below] at (5) {$e_{5}$};
\fill(6)circle (5pt);
\node[below] at (6) {$e_{6}$};

\draw[thick] (5)--(1);
\draw[thick] (5)--(2);
\draw[thick] (5)--(3);
\draw[thick] (5)--(4);
\draw[thick] (6)--(1);
\draw[thick] (6)--(2);
\draw[thick] (6)--(3);
\draw[thick] (6)--(4);

\draw[dashed](1)to[out=195,in=45](5);
\draw[dashed](1)to[out=225,in=75](6);
\draw[dashed](2)to[out=225,in=75](5);
\draw[dashed](2)to[out=255,in=105](6);
\draw[dashed](3)to[out=255,in=105](5);
\draw[dashed](3)to[out=285,in=135](6);
\draw[dashed](4)to[out=285,in=135](5);
\draw[dashed](4)to[out=315,in=165](6);

\end{scope}
\end{tikzpicture}
\caption{The weighted signed graph corresponding to $C_{6}$ and $C_{4}\times C_{2}$. The obtained graph does not have a subgraph which appeared in Figure~\ref{example_weighted_vertex} but have a subgraph in Figure~\ref{example_weighted_vertex2}.
Therefore, this is $3$-independent}
\label{example_typeC_Case2}
\end{figure}

This proves that $(C_{n},C_{i}\times C_{n-i})$ is one of the cases in Theorem~\ref{main} (2).

\section{Type $D$}
\label{sect:7}

In this section, we compute $k(G/H)$ for the case when $G$ is of type $D_{n}$ for $n\ge 4$.
In this case, the corresponding weighted signed graph of $G$ is $\pm K_{n}$.
We remark the identifications $D_{3}=A_{3}$, $D_{2}=A_{1}\times A_{1}$ and $D_{1}=T^{1}$.
This implies that the weighted signed graph $\pm K_{3}$ (resp.~$\pm K_{2}$) can be identified with the weighted signed graph $K_{4}$ (resp.~$K_{2}\sqcup K_{2}$) because their weights on edges are the same.

\subsection{Case 1: $H=D_{i}\times D_{n-i}$ $(1\le i\le n-1)$} 
\label{sect:7.1}

This case can be computed by removing the edges of the weighted signed graph $\pm K_{i}\sqcup \pm K_{n-i}$ from $\pm K_{n}$, where $\pm K_{1}=\emptyset$.
The obtained graph does not have a subgraph which appeared in Figure~\ref{example_weighted_vertex} but has a subgraph in Figure~\ref{example_weighted_vertex2} (see 
Figure~\ref{example_typeD_Case1} for $(n,i)=(4,1)$ and Figure~\ref{example_typeD_Case2} for $(n,i)=(6,2)$).
Therefore, the set of vectors $\Delta_{G,H}/\{\pm 1\}$ is $3$-independent.

\begin{figure}[H]
\begin{tikzpicture}
\begin{scope}[xscale=0.3, yscale=0.3]

\coordinate (A) at (5,0);
\coordinate (B) at ({5*cos(2*pi/4 r)},{5*sin(2*pi/4 r)});
\coordinate (C) at ({5*cos(4*pi/4 r)},{5*sin(4*pi/4 r)});
\coordinate (D) at ({5*cos(6*pi/4 r)},{5*sin(6*pi/4 r)});

\fill(A) circle (5pt);
\node[right] at (A) {$e_{1}$};
\fill(B)circle (5pt);
\node[above] at (B) {$e_{2}$};
\fill(C)circle (5pt);
\node[label={90:$e_{3}$}] at (C) {};
\fill(D)circle (5pt);
\node[left] at (D) {$e_{4}$};

\draw[dashed](A)to[out=120,in=-30](B); 
\draw[thick](A)to[out=150,in=-70](B); 
\draw[dashed](A)to[out=165,in=15](C); 
\draw[thick](A)to[out=195,in=-25](C); 
\draw[dashed](A)to[out=210,in=60](D);
\draw[thick](A)to[out=240,in=20](D); 

\draw[thick](B)to[out=260,in=30](C);
\draw[dashed](B)to[out=210,in=60](C);
\draw[dashed](B)to[out=255,in=105](D);
\draw[thick](B)to[out=285,in=75](D);
\draw[dashed](C)to[out=-30,in=105](D);
\draw[thick](C)to[out=-60,in=150](D);

\node[right] at (7,0) {$-$};

\coordinate (a) at (20,0);
\coordinate (b) at ({5*cos(2*pi/4 r)+15},{5*sin(2*pi/4 r)});
\coordinate (c) at ({5*cos(4*pi/4 r)+15},{5*sin(4*pi/4 r)});
\coordinate (d) at ({5*cos(6*pi/4 r)+15},{5*sin(6*pi/4 r)});

\fill(a) circle (5pt);
\node[right] at (a) {$e_{1}$};
\fill(b)circle (5pt);
\node[above] at (b) {$e_{2}$};
\fill(c)circle (5pt);
\node[label={90:$e_{3}$}] at (c) {};
\fill(d)circle (5pt);
\node[left] at (d) {$e_{4}$};

\draw[dashed](a)to[out=120,in=-30](b); 
\draw[thick](a)to[out=150,in=-70](b); 
\draw[dashed](a)to[out=165,in=15](c); 
\draw[thick](a)to[out=195,in=-25](c); 
\draw[thick](b)to[out=260,in=30](c);
\draw[dashed](b)to[out=210,in=60](c);

\node[right] at (22,0) {$=$};

\coordinate (1) at (35,0);
\coordinate (2) at ({5*cos(2*pi/4 r)+30},{5*sin(2*pi/4 r)});
\coordinate (3) at ({5*cos(4*pi/4 r)+30},{5*sin(4*pi/4 r)});
\coordinate (4) at ({5*cos(6*pi/4 r)+30},{5*sin(6*pi/4 r)});

\fill(1) circle (5pt);
\node[right] at (1) {$e_{1}$};
\fill(2)circle (5pt);
\node[above] at (2) {$e_{2}$};
\fill(3)circle (5pt);
\node[above] at (3) {$e_{3}$};
\fill(4)circle (5pt);
\node[left] at (4) {$e_{4}$};

\draw[dashed](1)to[out=210,in=60](4);
\draw[thick](1)to[out=240,in=20](4); 
\draw[dashed](2)to[out=255,in=105](4);
\draw[thick](2)to[out=285,in=75](4);
\draw[dashed](3)to[out=-30,in=105](4);
\draw[thick](3)to[out=-60,in=150](4);

\end{scope}
\end{tikzpicture}
\caption{The weighted signed graph corresponding to $D_{4}$ and $D_{3}\times T^1(=A_{3}\times T^1)$. This is $3$-independent (by Figure~\ref{example_weighted_vertex2})}
\label{example_typeD_Case1}
\end{figure}
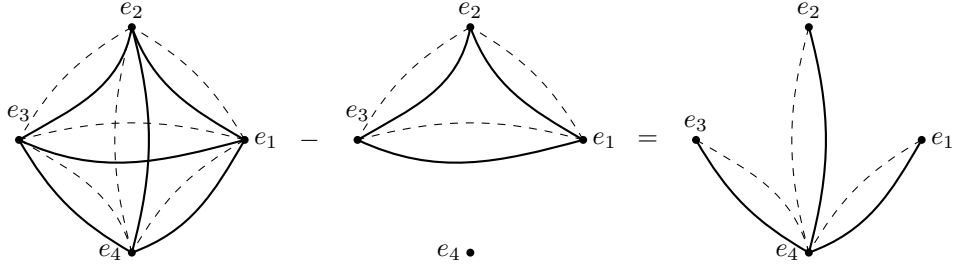

\begin{figure}[H]
\begin{tikzpicture}
\begin{scope}[xscale=0.3, yscale=0.3]

\coordinate (A) at (5,0);
\coordinate (B) at ({5*cos(2*pi/6 r)},{5*sin(2*pi/6 r)});
\coordinate (C) at ({5*cos(4*pi/6 r)},{5*sin(4*pi/6 r)});
\coordinate (D) at ({5*cos(6*pi/6 r)},{5*sin(6*pi/6 r)});
\coordinate (E) at ({5*cos(8*pi/6 r)},{5*sin(8*pi/6 r)});
\coordinate (F) at ({5*cos(10*pi/6 r)},{5*sin(10*pi/6 r)});

\fill(A) circle (5pt);
\node[right] at (A) {$e_{1}$};
\fill(B)circle (5pt);
\node[above] at (B) {$e_{2}$};
\fill(C)circle (5pt);
\node[above] at (C) {$e_{3}$};
\fill(D)circle (5pt);
\node[left] at (D) {$e_{4}$};
\fill(E)circle (5pt);
\node[below] at (E) {$e_{5}$};
\fill(F)circle (5pt);
\node[below] at (F) {$e_{6}$};

\draw[thick] (A)--(B)--(C)--(D)--(E)--(F)--(A);
\draw[thick] (A)--(C);
\draw[thick] (A)--(D);
\draw[thick] (A)--(E);
\draw[thick] (B)--(D);
\draw[thick] (B)--(E);
\draw[thick] (B)--(F);
\draw[thick] (C)--(D);
\draw[thick] (C)--(E);
\draw[thick] (C)--(F);
\draw[thick] (D)--(F);

\draw[dashed](A)to[out=105,in=-45](B); 
\draw[dashed](A)to[out=135,in=-15](C); 
\draw[dashed](A)to[out=165,in=15](D); 
\draw[dashed](A)to[out=195,in=45](E);
\draw[dashed](A)to[out=225,in=75](F);
\draw[dashed](B)to[out=165,in=15](C);
\draw[dashed](B)to[out=195,in=45](D);
\draw[dashed](B)to[out=225,in=75](E);
\draw[dashed](B)to[out=255,in=105](F);
\draw[dashed](C)to[out=225,in=75](D);
\draw[dashed](C)to[out=255,in=105](E);
\draw[dashed](C)to[out=285,in=135](F);
\draw[dashed](D)to[out=285,in=135](E);
\draw[dashed](D)to[out=315,in=165](F);
\draw[dashed](F)to[out=165,in=15](E); 

\node[right] at (7,0) {$-$};

\coordinate (a) at (20,0);
\coordinate (b) at ({5*cos(2*pi/6 r)+15},{5*sin(2*pi/6 r)});
\coordinate (c) at ({5*cos(4*pi/6 r)+15},{5*sin(4*pi/6 r)});
\coordinate (d) at ({5*cos(6*pi/6 r)+15},{5*sin(6*pi/6 r)});
\coordinate (e) at ({5*cos(8*pi/6 r)+15},{5*sin(8*pi/6 r)});
\coordinate (f) at ({5*cos(10*pi/6 r)+15},{5*sin(10*pi/6 r)});

\fill(a) circle (5pt);
\node[right] at (a) {$e_{1}$};
\fill(b)circle (5pt);
\node[above] at (b) {$e_{2}$};
\fill(c)circle (5pt);
\node[above] at (c) {$e_{3}$};
\fill(d)circle (5pt);
\node[left] at (d) {$e_{4}$};
\fill(e)circle (5pt);
\node[below] at (e) {$e_{5}$};
\fill(f)circle (5pt);
\node[below] at (f) {$e_{6}$};

\draw[thick] (a)--(b)--(c)--(d)--(a);
\draw[thick] (a)--(c);
\draw[thick] (b)--(d);
\draw[thick] (e)--(f);

\draw[dashed](a)to[out=105,in=-45](b); 
\draw[dashed](a)to[out=135,in=-15](c); 
\draw[dashed](a)to[out=165,in=15](d); 
\draw[dashed](b)to[out=165,in=15](c);
\draw[dashed](b)to[out=195,in=45](d);
\draw[dashed](c)to[out=225,in=75](d);

\draw[dashed](f)to[out=165,in=15](e); 

\node[right] at (22,0) {$=$};

\coordinate (1) at (36,0);
\coordinate (2) at ({5*cos(2*pi/6 r)+31},{5*sin(2*pi/6 r)});
\coordinate (3) at ({5*cos(4*pi/6 r)+31},{5*sin(4*pi/6 r)});
\coordinate (4) at ({5*cos(6*pi/6 r)+31},{5*sin(6*pi/6 r)});
\coordinate (5) at ({5*cos(8*pi/6 r)+31},{5*sin(8*pi/6 r)});
\coordinate (6) at ({5*cos(10*pi/6 r)+31},{5*sin(10*pi/6 r)});

\fill(1) circle (5pt);
\node[right] at (1) {$e_{1}$};
\fill(2)circle (5pt);
\node[above] at (2) {$e_{2}$};
\fill(3)circle (5pt);
\node[above] at (3) {$e_{3}$};
\fill(4)circle (5pt);
\node[left] at (4) {$e_{4}$};
\fill(5)circle (5pt);
\node[below] at (5) {$e_{5}$};
\fill(6)circle (5pt);
\node[below] at (6) {$e_{6}$};

\draw[thick] (5)--(1);
\draw[thick] (5)--(2);
\draw[thick] (5)--(3);
\draw[thick] (5)--(4);
\draw[thick] (6)--(1);
\draw[thick] (6)--(2);
\draw[thick] (6)--(3);
\draw[thick] (6)--(4);

\draw[dashed](1)to[out=195,in=45](5);
\draw[dashed](1)to[out=225,in=75](6);
\draw[dashed](2)to[out=225,in=75](5);
\draw[dashed](2)to[out=255,in=105](6);
\draw[dashed](3)to[out=255,in=105](5);
\draw[dashed](3)to[out=285,in=135](6);
\draw[dashed](4)to[out=285,in=135](5);
\draw[dashed](4)to[out=315,in=165](6);

\end{scope}
\end{tikzpicture}
\caption{The weighted signed graph corresponding to $D_{6}$ and $D_{2}\times D_{4}$. The obtained graph does not have a subgraph which appeared in Figure~\ref{example_weighted_vertex} but have a subgraph in Figure~\ref{example_weighted_vertex2}.
Therefore, this is $3$-independent}
\label{example_typeD_Case2}
\end{figure}
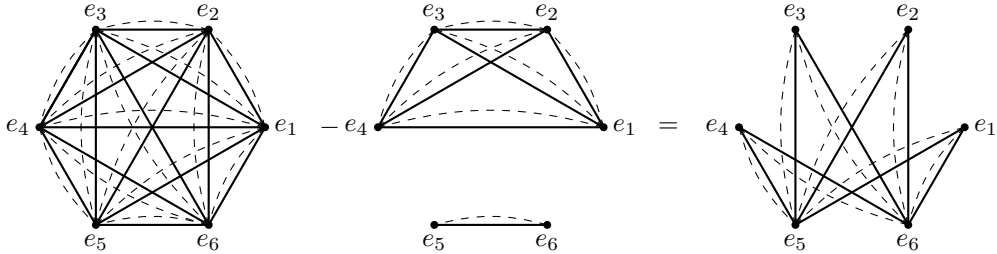

This proves that $(D_{n},D_{i}\times D_{n-i})$ is one of the cases in Theorem~\ref{main} (2).

\subsection{Case 2: $H=A_{n-1}\times T^{1}$} 
\label{sect:7.2}
This case can be computed by removing the edges of the weighted signed graph $K_{n}$ from $\pm K_{n}$.
The obtained graph is the complete graph with dashed edges.
This is $3$-independent because there is not a subgraph in Figure~\ref{example_weighted_vertex} but a subgraph in Figure~\ref{example_weighted_vertex2}.
(see Figure~\ref{example_typeD_Case3} for the case when $n=4$).
Therefore, the set of vectors $\Delta_{G,H}/\{\pm 1\}$ is $3$-independent.

\begin{figure}[H]
\begin{tikzpicture}
\begin{scope}[xscale=0.5, yscale=0.5]

\fill(8,-2)circle (5pt);
\node[left] at (8,-2) {$e_{1}$};
\fill(12,-2)circle (5pt);
\node[right] at (12,-2) {$e_{2}$};
\fill(8,2)circle (5pt);
\node[left] at (8,2) {$e_{4}$};
\fill(12,2)circle (5pt);
\node[right] at (12,2) {$e_{3}$};

\draw[dashed](8,-2)to[out=15,in=165](12,-2);
\draw[thick](8,-2)to[out=-15,in=-165](12,-2);
\draw[thick](8,2)to[out=15,in=165](12,2);
\draw[dashed](8,2)to[out=-15,in=-165](12,2);
\draw[thick](12,-2)to[out=75,in=-75](12,2);
\draw[dashed](12,-2)to[out=105,in=-105](12,2);
\draw[thick](8,-2)to[out=105,in=-105](8,2);
\draw[dashed](8,-2)to[out=75,in=-75](8,2);
\draw[thick](8,-2)to[out=60,in=-150](12,2);
\draw[dashed](8,-2)to[out=30,in=-105](12,2);
\draw[thick](8,2)to[out=-30,in=120](12,-2);
\draw[dashed](8,2)to[out=-60,in=150](12,-2);

\node[right] at (13,0) {$-$};

\fill(15,-2)circle (5pt);
\node[left] at (15,-2) {$e_{1}$};
\fill(19,-2)circle (5pt);
\node[right] at (19,-2) {$e_{2}$};
\fill(15,2)circle (5pt);
\node[left] at (15,2) {$e_{4}$};
\fill(19,2)circle (5pt);
\node[right] at (19,2) {$e_{3}$};

\draw[thick](15,-2)to[out=-15,in=-165](19,-2);
\draw[thick](15,2)to[out=15,in=165](19,2);
\draw[thick](19,-2)to[out=75,in=-75](19,2);
\draw[thick](15,-2)to[out=105,in=-105](15,2);
\draw[thick](15,-2)to[out=60,in=-150](19,2);
\draw[thick](15,2)to[out=-30,in=120](19,-2);

\node[right] at (20,0) {$=$};

\fill(22,-2)circle (5pt);
\node[left] at (22,-2) {$e_{1}$};
\fill(26,-2)circle (5pt);
\node[right] at (26,-2) {$e_{2}$};
\fill(22,2)circle (5pt);
\node[left] at (22,2) {$e_{4}$};
\fill(26,2)circle (5pt);
\node[right] at (26,2) {$e_{3}$};

\draw[dashed](22,-2)to[out=15,in=165](26,-2);
\draw[dashed](22,2)to[out=-15,in=-165](26,2);
\draw[dashed](26,-2)to[out=105,in=-105](26,2);
\draw[dashed](22,-2)to[out=75,in=-75](22,2);
\draw[dashed](22,-2)to[out=30,in=-105](26,2);
\draw[dashed](22,2)to[out=-60,in=150](26,-2);

\end{scope}
\end{tikzpicture}
\caption{The weighted signed graph corresponding to $D_{4}$ and $A_{3}\times T^{1}(=D_{3}\times T^{1})$. This is $3$-independent. }
\label{example_typeD_Case3}
\end{figure}
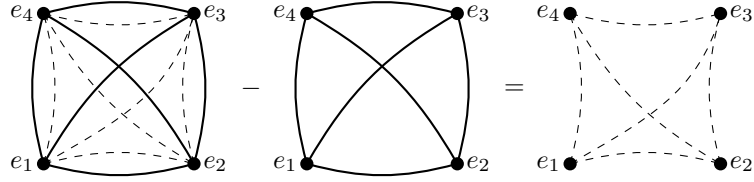

\begin{remark}
Note that if $n=4$, then we may regard 
the weighted signed graph of $H(=A_{3}\times T^1)$ as $K_{4}$ (see Figure~\ref{example_typeD_Case3}).
On the other hand, there is the identification $A_{3}=D_{3}$. 
In this case, we may also regard the weighted signed graph of $H(=D_{3}\times T^1)$ as (see Figure~\ref{example_typeD_Case1}).
Namely, 
if we regard $H=D_{3}\times T^1$, this case also can be computed from Case 1.
\end{remark}

This proves that $(D_{n},A_{n-1}\times T^{1})$ is one of the cases in Theorem~\ref{main} (2).

\section{Type $F$}
\label{sect:8}

We compute $k(G/H)$ for the case when $G$ is of type $F_{4}$.
We provide the reader with the explicit description of positive roots of $F_{4}$ in Table~\ref{fig:f4} following~\cite[pp. 58--59]{Y25}.
Following~\cite{Y25}, we denote each vector in Table~\ref{fig:f4} as $(n_{1}n_{2}n_{3}n_{4})$ or $\sum_{i=1}^{4}n_{i}\alpha_{i}$.
\begin{table}[H]
\centering
\caption{Positive $24$ roots of $F_{4}$ system}\label{fig:f4}
\begin{subtable}{0.4\textwidth}
\centering
\begin{tabular}{llll}
1 & 1 & 1 & 0 \\
0 & 1 & 1 & 0 \\
0 & 0 & 1 & 0 \\
1 & 2 & 3 & 2 \\
1 & 0 & 0 & 0 \\
1 & 1 & 0 & 0 \\
0 & 1 & 2 & 2 \\
0 & 1 & 0 & 0 \\
1 & 1 & 2 & 2 \\
1 & 2 & 2 & 2 \\
1 & 2 & 2 & 0 \\
1 & 1 & 2 & 0
\end{tabular}
\end{subtable}
\begin{subtable}{0.4\textwidth}
\centering
\begin{tabular}{lllll}
          & 2 & 3 & 4 & 2 \\
          & 0 & 1 & 2 & 0 \\
          & 1 & 3 & 4 & 2 \\
          & 1 & 2 & 4 & 2 \\
$x_{1}:$ & 1 & 2 & 3 & 1 \\
$x_{2}:$ & 0 & 0 & 0 & 1 \\
$x_{3}:$ & 1 & 2 & 2 & 1 \\
$x_{4}:$ & 1 & 1 & 2 & 1 \\
$x_{5}:$ & 0 & 0 & 1 & 1 \\
$x_{6}:$ & 0 & 1 & 1 & 1 \\
$x_{7}:$ & 1 & 1 & 1 & 1 \\
$x_{8}:$ & 0 & 1 & 2 & 1
\end{tabular}
\end{subtable}
\end{table}

\subsection{Case 1: $H=C_{3}\times A_{1}$}
\label{sect:8.1}

The positive root system embedding of $C_{3}\times A_{1}$ to $F_{4}$ can be obtained by removing $\alpha_{1}=(1000)$ from the extended Dynkin diagram for $F_{4}$~\cite[p.58]{Y25}. 
The respective subroot system $C_{3}\times A_{1}$ (of cardinality $10$) is the subset of the set in Table~\ref{fig:f4} given by the condition that $n_{1}$ is even.
The respective complement $\Delta_{G,H}^{+}$ is given by the set of all vectors with $n_{1}=1$ in Table~\ref{fig:f4}, 
and it has a linearly dependent triple
\[
(1120)+(1100)=2\cdot (1110).
\]
By definition of the root systems, each pair of positive roots are linearly independent.
Therefore, $k(G/H)=2$, and this proves that $(F_{4},C_{3}\times A_{1})$ is one of the cases in Theorem~\ref{main} (3).

\subsection{Case 2: $H=B_{4}$}
\label{sect:8.2}
The positive root system embedding of $B_{4}$ to $F_{4}$ 
can be obtained by removing 
$\alpha_{4}=(0001)$ from the extended Dynkin diagram for $F_{4}$~\cite[p.58]{Y25}.
The respective subroot system $B_{4}$ (of cardinality $16$) is the subset of the set in Table~\ref{fig:f4} given by the condition that $n_{4}$ is even.
The respective complement $\Delta_{G,H}^{+}$ is given by the set of all vectors with $n_{4}=1$ in Table~\ref{fig:f4}.
We denote this set by
\[
\Delta_{G,H}^{+} = \{x_1, \ldots, x_8\},
\]
where the elements are ordered from top to bottom as they appear in Table~\ref{fig:f4} with $n_{4}=1$.
It has a linearly dependent quadruple $x_{1}+x_{2}=x_{3}+x_{5}$, i.e., 
\[
(1231)+(0001)=(1221)+(0011).
\]
Therefore, $k(G/H)=2$ or $3$.

To prove $k(G/H)\not=2$, we need to check for each triple of pairwise distinct roots $\{u,v,w\}$ in $\Delta_{G,H}^{+}$ are linearly independent. 
It is well known that 
$\Delta_{F_{4},B_{4}}=\Delta_{G,H}^{+}\sqcup -\Delta_{G,H}^{+}=\{\pm x_{1},\ldots, \pm x_{8}\}$
may be identified with  
the following collection of $16$ roots in $\R^{4}$ 
(see~\cite[p59]{Y25} or~\cite{B}):
\begin{equation}\label{eq:halfsum}
\frac{1}{2}(a_{1},a_{2},a_{3},a_{4}),\ a_{i}\in\lb -1,1\rb.
\end{equation}
Hence, for any 
distinct positive
roots $u,v\in \Delta_{G,H}^{+}$ the scalar product takes values
\[
\la u,v\ra\in\lb -\frac{1}{2},0,\frac{1}{2}\rb.
\]
Given any triple $u_{1}, u_{2}, u_{3}\in \Delta_{G,H}^{+}$ of pairwise distinct roots, consider its respective Gram matrix
\[
G=(\la u_{i}, u_{j}\ra):=
\begin{pmatrix}
1 & a & b\\
a & 1 & c\\
b & c & 1
\end{pmatrix},\ 
a,b,c\in\lb -\frac{1}{2},0,\frac{1}{2}\rb.
\]
We compute
\[
\det G=1+2abc-a^2-b^2-c^2.
\]
Since
\[
a^2+b^2+c^2\leq \frac{3}{4},\ |2abc|\leq \frac{1}{4},
\]
one has 
\[
\det G\geq 0,
\]
with equality possible if and only if
\[
abc<0,\ |a|=|b|=|c|=\frac{1}{2}.
\]
Assume that the triple $u,v,w\in \Delta^{+}_{G,H}$ of pairwise distinct roots is linearly dependent.
Since $\det G=0$ for their Gram matrix $G$ if and only if $u,v,w$ are linearly dependent, the conditions in the last formula hold for $u,v,w$.
However, if $|\la u, v\ra|=|\la u, w\ra|=\frac{1}{2}$, then it follows from \eqref{eq:halfsum} that $\la v, w\ra=0$.
This gives a contradiction to $|\la v, w\ra|=\frac{1}{2}$.
This proves $k(G/H)\not=2$. Therefore, $(F_{4},B_{4})$ is one of the cases in Theorem~\ref{main} $(2)$.


\section{Type $E$ cases}
\label{sect:9}

Let $G=E_{6},E_{7},E_{8}$. 
In this remaining section, we prove $k(G/H)=3$ for all seven cases $(G,H)$ which appeared in Theorem~\ref{thm:cart}. We first claim the following lemma.

\begin{lemma}\label{lm:triprk2}
Let $G$ be of type $A, D, E$, i.e., the simply laced, simple Lie group.
Any linearly dependent positive root triple 
$u,v,w\in \Delta_{G}^{+}$ satisfies
\[
u=v+w,
\]
up to permutation of the triple.
\end{lemma}
\begin{proof}
It is well known that the span of any two (linearly independent) roots in any root system $\Delta_{G}$ is a root subsystem (of rank $2$)~\cite[Section 9.4]{H72}.
After multiplication by $-1$ if necessary, without loss of generality assume that $u,v,w$ belong to a positive root cone.
There is the classification for rank $2$ root systems~\cite[Section 9.4]{H72}, and the respective positive roots for a basis $\alpha,\beta$ of simple roots is given as follows:
\begin{description}
\item[$A_{1}\times A_{1}$] $\alpha$, $\beta$;
\item[$A_{2}$] $\alpha$, $\beta$, $\alpha+\beta$;
\item[$B_{2}=C_{2}$] $\alpha$, $\beta$, $\alpha+\beta$, $\alpha+2\beta$;
\item[$G_{2}$] $\alpha$, $\beta$, $\alpha+\beta$, $\alpha+2\beta$, $\alpha+3\beta$, $2\alpha+3\beta$.
\end{description}
This shows that these are all possible linearly dependent triples for a rank $2$ root system.
Only $A_{1}\times A_{1}$, $A_{2}$ from the above list are simply laced.
Therefore, $u,v,w\in \Delta_{G}^{+}$ must be of the form in the statement.
\end{proof}


\subsection{Type $E_{6}$}
\label{sect:9.1}
Assume that $G=E_{6}$. 
We prove $k(G/H)=3$ for the cases when $H=A_{5}\times A_{1}$ and $H=D_{5}\times T^{1}$ (see Theorem~\ref{thm:cart}).
The following Table~\ref{fig:e6} is the explicit description of positive roots $n_{1}n_{2}\cdots n_{6}$ (also denote $(n_{1}n_{2}\cdots n_{6})$, $\sum_{i=1}^{6} n_{i}\alpha_{i}$) of type $E_{6}$ (see~\cite[93--94]{Y25}):
\begin{table}[H]
\centering
\caption{Positive $36$ roots of $E_{6}$ system}\label{fig:e6}
\begin{subtable}{0.4\textwidth}
\centering
\begin{tabular}{llllll}
1 & 0 & 0 & 0 & 0 & 0 \\
1 & 1 & 0 & 0 & 0 & 0 \\
1 & 1 & 1 & 0 & 0 & 0 \\
0 & 1 & 0 & 0 & 0 & 0 \\
0 & 1 & 1 & 0 & 0 & 0 \\
0 & 0 & 1 & 0 & 0 & 0 \\
1 & 2 & 2 & 1 & 1 & 1 \\
1 & 1 & 2 & 1 & 1 & 1 \\
1 & 1 & 1 & 1 & 1 & 1 \\
0 & 1 & 2 & 1 & 1 & 1 \\
0 & 1 & 1 & 1 & 1 & 1 \\
0 & 0 & 1 & 1 & 1 & 1 \\
1 & 1 & 1 & 1 & 0 & 0 \\
0 & 1 & 1 & 1 & 0 & 0 \\
0 & 0 & 1 & 1 & 0 & 0 \\
0 & 0 & 0 & 1 & 0 & 0 \\
1 & 1 & 1 & 0 & 1 & 1 \\
0 & 1 & 1 & 0 & 1 & 1 \\
\end{tabular}
\end{subtable}
\centering
\begin{subtable}{0.4\textwidth}
\centering
\begin{tabular}{llllll}
0 & 0 & 1 & 0 & 1 & 1 \\
0 & 0 & 0 & 0 & 1 & 1 \\
0 & 0 & 1 & 0 & 1 & 0 \\
1 & 2 & 3 & 1 & 2 & 1 \\
0 & 1 & 2 & 1 & 2 & 1 \\
1 & 1 & 2 & 1 & 2 & 1 \\
1 & 2 & 2 & 1 & 2 & 1 \\
0 & 0 & 0 & 0 & 1 & 0 \\
1 & 1 & 1 & 0 & 1 & 0 \\
0 & 1 & 1 & 0 & 1 & 0 \\
1 & 1 & 2 & 1 & 1 & 0 \\
1 & 1 & 1 & 1 & 1 & 0 \\
0 & 0 & 0 & 0 & 0 & 1 \\
1 & 2 & 2 & 1 & 1 & 0 \\
0 & 1 & 1 & 1 & 1 & 0 \\
0 & 1 & 2 & 1 & 1 & 0 \\
1 & 2 & 3 & 2 & 2 & 1 \\
0 & 0 & 1 & 1 & 1 & 0
\end{tabular}
\end{subtable}
\end{table}
\begin{figure}[H]
\centering
\begin{tikzpicture}
\begin{scope}[xscale=0.25, yscale=0.25]
\fill(-8,1)circle (5pt);
\node[above] at (-8,1) {$1$};
\node[below] at (-8,1) {$\alpha_{1}$};
\fill(-4,1)circle (5pt);
\node[above] at (-4,1) {$2$};
\node[below] at (-4,1) {$\alpha_{2}$};
\fill(0,1)circle (5pt);
\node[above] at (0,1) {$3$};
\node[below] at (-1,1) {$\alpha_{3}$};
\fill(0,-3)circle (5pt);
\node[left] at (0,-3) {$2$};
\node[right] at (0,-3) {$\alpha_{4}$};
\fill(4,1)circle (5pt);
\node[above] at (4,1) {$2$};
\node[below] at (4,1) {$\alpha_{5}$};
\fill(8,1)circle (5pt);
\node[above] at (8,1) {$1$};
\node[below] at (8,1) {$\alpha_{6}$};
\fill(0,-7)circle (5pt);
\node[right] at (0,-7) {$-\widetilde{\alpha}$};

\draw[thick](-8,1)to(-4,1);
\draw[thick](-4,1)to(0,1);
\draw[thick](0,1)to(0,-3);
\draw[thick](0,1)to(4,1);
\draw[thick](4,1)to(8,1);

\draw[thick](0,-3)to(0,-7);

\end{scope}
\end{tikzpicture}
\caption{Extended Dynkin diagram of $E_{6}$ root system, where the dominant root $\widetilde{\alpha}=\sum_{i=1}^{6}m_{i}\alpha_{i}=(123221)$ and the integers represent $m_{i}$'s in Theorem~\ref{thm:BdS}}\label{fig:e6dd}
\end{figure}
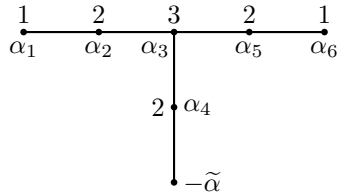

\subsubsection{Case 1: $(G,H)=(E_{6},A_{5}\times A_{1})$}
\label{sect:9.1.1}
In this case, according to Borel-de Siebenthal theory,
there are three different positive root system embeddings of $A_{5}\times A_{1}$ to $E_{6}$, corresponding to a vertex $\alpha_{i}$, $i=2,4,5$, of the extended Dynkin diagram (Figure~\ref{fig:e6dd}) with $m_{i}=2$.
Each embedding (of cardinality $16$) is given by the condition that $n_{i}$ is even in Table~\ref{fig:e6}, where $i=2,4,5$, respectively.
This implies that the vectors in the respective ($i=2,4,5$) complement satisfy $n_{i}=1$ in Table~\ref{fig:e6}.
Each case has the following linearly dependent quadruple in $\Delta^{+}_{G,H}$:
\[
(112111)-(111111)=(012111)-(011111).
\]
Therefore, $k(G/H)=2,3$.
 
To prove $k(G/H)\not=2$, it is enough to prove that each triple $u,v,w\in \Delta^{+}_{G,H}$ is linearly independent.
By Lemma~\ref{lm:triprk2},
if there is a linearly dependent triple $u,v,w\in \Delta^{+}_{G,H}$, then we may write $u=v+w$.
However, $\Delta^{+}_{G,H}$ consists of the roots with $n_{i}=1$ (for $i=2,4$ or $5$).
Therefore, for $v,w\in \Delta^{+}_{G,H}$, $u=v+w$ satisfies that $n_{i}=2$ and $u\not\in \Delta^{+}_{G,H}$.
This gives a contradiction.
Therefore, $k(G/H)=3$ and $(E_{6},A_{5}\times A_{1})$ is one of the cases in Theorem~\ref{main} (2).

\subsubsection{Case 2: $(G,H)=(E_{6},D_{5}\times T^{1})$}
\label{sect:9.1.2}
In this case, according to Borel-deSiebenthal theory,
there are two different positive root system embeddings of $D_{5}\times T^{1}$ to $E_{6}$, corresponding to a vertex $\alpha_{i}$, $i=1,6$, of the Dynkin diagram for $E_{6}$ (the diagram removing $-\widetilde{\alpha}$ from Figure~\ref{fig:e6dd}) with $m_{i}=1$.
Each embedding (of cardinality $20$) is given by the condition that $n_{i}=0$ in Table~\ref{fig:e6}, where $i=1,6$, respectively.
This implies that the vectors in the respective ($i=1,6$) complement satisfy $n_{i}=1$ in Table~\ref{fig:e6}.
Each case has the following linearly dependent quadruple in $\Delta^{+}_{G,H}$:
\[
(122111)-(112111)=(122121)-(112121).
\]
Therefore, $k(G/H)=2,3$.
By an argument similar to that used in the case $(G,H)=(E_{6},A_{5}\times A_{1})$, we also obtain $k(G/H)\not=2$.
Therefore, $k(G/H)=3$ and $(E_{6},D_{5}\times T^{1})$ is one of the cases in Theorem~\ref{main} (2).

\subsection{Type $E_{7}$}
\label{sect:9.2}

Assume that $G=E_{7}$. 
We prove $k(G/H)=3$ for the cases when $H=A_{7}$, $D_{6}\times A_{1}$ and $H=E_{6}\times T^{1}$ (see Theorem~\ref{thm:cart}).
The following Table~\ref{fig:e7} is the explicit description of positive roots $n_{1}n_{2}\cdots n_{7}$ (also denote $(n_{1}n_{2}\cdots n_{7})$, $\sum_{i=1}^{7}n_{i}\alpha_{i}$) of type $E_{7}$ (see~\cite[141--144]{Y25}):
\begin{table}[H]
\centering
\caption{Positive $63$ roots of $E_{7}$ system}\label{fig:e7}
\begin{subtable}{0.3\textwidth}
\centering
\begin{tabular}{lllllll}
1 & 0 & 0 & 0 & 0 & 0 & 0 \\
1 & 1 & 0 & 0 & 0 & 0 & 0 \\
1 & 1 & 1 & 0 & 0 & 0 & 0 \\
0 & 1 & 0 & 0 & 0 & 0 & 0 \\
0 & 1 & 1 & 0 & 0 & 0 & 0 \\
0 & 0 & 1 & 0 & 0 & 0 & 0 \\
1 & 2 & 2 & 2 & 2 & 1 & 1 \\
1 & 1 & 2 & 2 & 2 & 1 & 1 \\
1 & 1 & 1 & 2 & 2 & 1 & 1 \\
0 & 1 & 2 & 2 & 2 & 1 & 1 \\
0 & 1 & 1 & 2 & 2 & 1 & 1 \\
0 & 0 & 1 & 2 & 2 & 1 & 1 \\
1 & 1 & 1 & 1 & 1 & 1 & 1 \\
0 & 1 & 1 & 1 & 1 & 1 & 1 \\
0 & 0 & 1 & 1 & 1 & 1 & 1 \\
0 & 0 & 0 & 1 & 1 & 1 & 1 \\
1 & 1 & 1 & 1 & 1 & 0 & 0 \\
0 & 1 & 1 & 1 & 1 & 0 & 0 \\
0 & 0 & 1 & 1 & 1 & 0 & 0 \\
0 & 0 & 0 & 1 & 1 & 0 & 0 \\
0 & 0 & 1 & 1 & 0 & 0 & 0
\end{tabular}
\end{subtable}
\centering
\begin{subtable}{0.3\textwidth}
\centering
\begin{tabular}{lllllll}
1 & 2 & 3 & 3 & 2 & 1 & 1 \\
0 & 1 & 2 & 3 & 2 & 1 & 1 \\
1 & 1 & 2 & 3 & 2 & 1 & 1 \\
1 & 2 & 2 & 3 & 2 & 1 & 1 \\
0 & 0 & 0 & 1 & 0 & 0 & 0 \\
1 & 1 & 1 & 1 & 0 & 0 & 0 \\
0 & 1 & 1 & 1 & 0 & 0 & 0 \\
1 & 1 & 2 & 2 & 1 & 1 & 1 \\
1 & 1 & 1 & 2 & 1 & 1 & 1 \\
0 & 0 & 0 & 0 & 1 & 0 & 0 \\
1 & 2 & 2 & 2 & 1 & 1 & 1 \\
0 & 1 & 1 & 2 & 1 & 1 & 1 \\
0 & 1 & 2 & 2 & 1 & 1 & 1 \\
1 & 2 & 3 & 4 & 3 & 2 & 2 \\
0 & 0 & 1 & 2 & 1 & 1 & 1 \\
1 & 2 & 3 & 4 & 2 & 1 & 2 \\
0 & 0 & 0 & 0 & 0 & 1 & 0 \\
0 & 0 & 0 & 0 & 0 & 0 & 1 \\
1 & 1 & 1 & 1 & 1 & 1 & 0 \\
0 & 1 & 1 & 1 & 1 & 1 & 0 \\
0 & 0 & 1 & 1 & 1 & 1 & 0
\end{tabular}
\end{subtable}
\centering
\begin{subtable}{0.3\textwidth}
\centering
\begin{tabular}{lllllll}
0 & 0 & 0 & 1 & 1 & 1 & 0 \\
1 & 1 & 1 & 1 & 1 & 0 & 1 \\
0 & 1 & 1 & 1 & 1 & 0 & 1 \\
0 & 0 & 1 & 1 & 1 & 0 & 1 \\
0 & 0 & 0 & 1 & 1 & 0 & 1 \\
0 & 0 & 1 & 1 & 0 & 0 & 1 \\
1 & 2 & 3 & 3 & 2 & 1 & 2 \\
0 & 1 & 2 & 3 & 2 & 1 & 2 \\
1 & 1 & 2 & 3 & 2 & 1 & 2 \\
1 & 2 & 2 & 3 & 2 & 1 & 2 \\
0 & 0 & 0 & 1 & 0 & 0 & 1 \\
1 & 1 & 1 & 1 & 0 & 0 & 1 \\
0 & 1 & 1 & 1 & 0 & 0 & 1 \\
1 & 1 & 2 & 2 & 1 & 0 & 1 \\
1 & 1 & 1 & 2 & 1 & 0 & 1 \\
0 & 0 & 0 & 0 & 1 & 1 & 0 \\
1 & 2 & 2 & 2 & 1 & 0 & 1 \\
0 & 1 & 1 & 2 & 1 & 0 & 1 \\
0 & 1 & 2 & 2 & 1 & 0 & 1 \\
1 & 2 & 3 & 4 & 3 & 1 & 2 \\
0 & 0 & 1 & 2 & 1 & 0 & 1
\end{tabular}
\end{subtable}
\end{table}
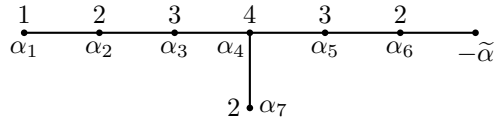
\begin{figure}[H]
\centering
\begin{tikzpicture}
\begin{scope}[xscale=0.25, yscale=0.25]
\fill(-8,1)circle (5pt);
\node[above] at (-8,1) {$1$};
\node[below] at (-8,1) {$\alpha_{1}$};
\fill(-4,1)circle (5pt);
\node[above] at (-4,1) {$2$};
\node[below] at (-4,1) {$\alpha_{2}$};
\fill(0,1)circle (5pt);
\node[above] at (0,1) {$3$};
\node[below] at (0,1) {$\alpha_{3}$};
\fill(4,1)circle (5pt);
\node[above] at (4,1) {$4$};
\node[below] at (3,1) {$\alpha_{4}$};
\fill(4,-3)circle (5pt);
\node[left] at (4,-3) {$2$};
\node[right] at (4,-3) {$\alpha_{7}$};
\fill(8,1)circle (5pt);
\node[above] at (8,1) {$3$};
\node[below] at (8,1) {$\alpha_{5}$};
\fill(12,1)circle (5pt);
\node[above] at (12,1) {$2$};
\node[below] at (12,1) {$\alpha_{6}$};
\fill(16,1)circle (5pt);
\node[below] at (16,1) {$-\widetilde{\alpha}$};

\draw[thick](-8,1)to(-4,1);
\draw[thick](-4,1)to(0,1);
\draw[thick](0,1)to(4,1);
\draw[thick](4,1)to(4,-3);
\draw[thick](4,1)to(8,1);
\draw[thick](8,1)to(12,1);
\draw[thick](12,1)to(16,1);

\end{scope}
\end{tikzpicture}
\caption{Extended Dynkin diagram of $E_{7}$ root system, where the dominant root $\widetilde{\alpha}=\sum_{i=1}^{7}m_{i}\alpha_{i}=(1234322)$ and the integers represent $m_{i}$'s in Theorem~\ref{thm:BdS}}\label{fig:e7dd}
\end{figure}

\subsubsection{Case 1: $(G,H)=(E_{7},D_6\times A_{1})$}
\label{sect:9.2.1}
In this case, according to Borel-deSiebenthal theory,
there are two positive root system embeddings of $D_{6}\times A_{1}$ (of cardinality $31$) to $E_{7}$, corresponding to a vertex $\alpha_{i}$, $i=2,6$ of the extended Dynkin diagram (Figure~\ref{fig:e7dd}) with $m_{i}=2$.
Each embedding (of cardinality $31$) is given by the condition that $n_{i}$ is even in Table~\ref{fig:e7}, where $i=2,6$, respectively.
This implies that the vectors in the respective ($i=2,6$) complement satisfy $n_{i}=1$ in Table~\ref{fig:e7}.
Each case has the following linearly dependent quadruple in $\Delta^{+}_{G,H}$:
\[
(1122211)-(1112211)=(0122211)-(0112211).
\]
Therefore, $k(G/H)=2,3$.
By an argument similar to that used in the case $(G,H)=(E_{6},A_{5}\times A_{1})$, we also obtain $k(G/H)\not=2$.
Therefore, $k(G/H)=3$ and $(E_{7},D_{6}\times A_{1})$ is one of the cases in Theorem~\ref{main} (2).

\subsubsection{Case 2: $(G,H)=(E_{7},A_{7})$}
\label{sect:9.2.2}
In this case, according to Borel-deSiebenthal theory,
the positive root system embedding of $A_{7}$ (of cardinality $28$) to $E_{7}$ corresponds to the vertex $\alpha_{7}$ of the extended Dynkin diagram  (Figure~\ref{fig:e7dd}) with $m_{7}=2$.
This is given by the set of roots such that $n_{7}$ is even in Table~\ref{fig:e7}. 
Therefore, $\Delta_{G,H}^{+}$ consists of the roots such that $n_{7}=1$ in Table~\ref{fig:e7}.
One has the linearly dependent quadruple in $\Delta^{+}_{G,H}$ from the previous case.
By an argument similar to that used in the case $(G,H)=(E_{6},A_{5}\times A_{1})$, we also obtain $k(G/H)\not=2$.
Therefore, $k(G/H)=3$ and $(E_{7},A_{7})$ is one of the cases in Theorem~\ref{main} (2).

\subsubsection{Case 3: $(G,H)=(E_{7},E_{6}\times T^{1})$}
\label{sect:9.2.3}
In this case, according to Borel-deSiebenthal theory,
the positive root system embedding of $E_{6}\times T^{1}$ (of cardinality $36$) to $E_{7}$ corresponds to the vertex $\alpha_{1}$ of the Dynkin diagram (the diagram removing $-\widetilde{\alpha}$ from Figure~\ref{fig:e7dd}) with $m_{1}=1$.
This is given by the set of roots such that $n_{1}=0$ in Table~\ref{fig:e7}. 
Therefore, $\Delta_{G,H}^{+}$ consists of the roots such that $n_{1}=1$ in Table~\ref{fig:e7}.
Then, $\Delta^{+}_{G,H}$ has the following linearly dependent quadruple:
\[
(1122211)-(1112211)=(1122111)-(1112111).
\]
Hence, $k(G/H)=2,3$.
By an argument similar to that used in the case $(G,H)=(E_{6},A_{5}\times A_{1})$, we also obtain $k(G/H)\not=2$.
Therefore, $k(G/H)=3$ and $(E_{7},E_{6}\times T^{1})$ is one of the cases in Theorem~\ref{main} (2).

\subsection{Type $E_{8}$}
\label{sect:9.3}
Assume that $G=E_{8}$. 
We prove $k(G/H)=3$ for the cases when $H=D_{8}$ and $H=E_{7}\times A_{1}$ (see Theorem~\ref{thm:cart}).
The following Table~\ref{fig:e8} is the explicit description of positive roots $n_{1}n_{2}\cdots n_{8}$ (also denote $(n_{1}n_{2}\cdots n_{8})$, $\sum_{i=1}^{8}n_{i}\alpha_{i}$) of type $E_{8}$ (see~\cite[196--202]{Y25}):
\begin{table}[h]
\centering
\caption{Positive $120$ roots of $E_{8}$ system}\label{fig:e8}
\begin{subtable}{0.3\textwidth}
\centering
\begin{tabular}{llllllll}
2 & 3 & 4 & 3 & 2 & 1 & 0 & 2 \\
2 & 3 & 4 & 3 & 2 & 1 & 1 & 2 \\
2 & 3 & 4 & 3 & 2 & 2 & 1 & 2 \\
0 & 0 & 0 & 0 & 0 & 0 & 1 & 0 \\
0 & 0 & 0 & 0 & 0 & 1 & 1 & 0 \\
0 & 0 & 0 & 0 & 0 & 1 & 0 & 0 \\
2 & 4 & 6 & 5 & 4 & 3 & 2 & 3 \\
2 & 4 & 6 & 5 & 4 & 3 & 1 & 3 \\
2 & 4 & 6 & 5 & 4 & 2 & 1 & 3 \\
0 & 1 & 2 & 2 & 2 & 2 & 1 & 1 \\
0 & 1 & 2 & 2 & 2 & 1 & 1 & 1 \\
0 & 1 & 2 & 2 & 2 & 1 & 0 & 1 \\
2 & 4 & 6 & 5 & 3 & 2 & 1 & 3 \\
0 & 1 & 2 & 2 & 1 & 1 & 1 & 1 \\
0 & 1 & 2 & 2 & 1 & 1 & 0 & 1 \\
0 & 1 & 2 & 2 & 1 & 0 & 0 & 1 \\
2 & 3 & 4 & 3 & 3 & 2 & 1 & 2 \\
0 & 0 & 0 & 0 & 1 & 1 & 1 & 0 \\
0 & 0 & 0 & 0 & 1 & 1 & 0 & 0 \\
0 & 0 & 0 & 0 & 1 & 0 & 0 & 0 \\
1 & 1 & 2 & 2 & 2 & 2 & 1 & 1 \\
1 & 1 & 2 & 2 & 2 & 1 & 1 & 1 \\
1 & 1 & 2 & 2 & 2 & 1 & 0 & 1 \\
1 & 0 & 0 & 0 & 0 & 0 & 0 & 0 \\
1 & 3 & 4 & 3 & 2 & 2 & 1 & 2 \\
1 & 3 & 4 & 3 & 2 & 1 & 1 & 2 \\
1 & 3 & 4 & 3 & 2 & 1 & 0 & 2 \\
1 & 2 & 2 & 1 & 0 & 0 & 0 & 1 \\
1 & 2 & 2 & 1 & 1 & 1 & 0 & 1 \\
1 & 2 & 2 & 1 & 1 & 0 & 0 & 1 \\
1 & 2 & 2 & 1 & 1 & 1 & 1 & 1 \\
1 & 3 & 4 & 3 & 3 & 2 & 1 & 2 \\
1 & 1 & 2 & 2 & 1 & 1 & 0 & 1 \\
1 & 1 & 2 & 2 & 1 & 0 & 0 & 1 \\
1 & 1 & 2 & 2 & 1 & 1 & 1 & 1 \\
1 & 2 & 4 & 4 & 3 & 2 & 1 & 2 \\
0 & 1 & 1 & 0 & 0 & 0 & 0 & 1 \\
0 & 0 & 1 & 1 & 0 & 0 & 0 & 1 \\
0 & 1 & 1 & 1 & 0 & 0 & 0 & 0 \\
2 & 4 & 5 & 4 & 3 & 2 & 1 & 2
\end{tabular}
\end{subtable}
\begin{subtable}{0.3\textwidth}
\centering
\begin{tabular}{llllllll}
2 & 3 & 5 & 4 & 3 & 2 & 1 & 3 \\
0 & 1 & 1 & 1 & 1 & 1 & 1 & 0 \\
0 & 0 & 1 & 1 & 1 & 1 & 1 & 1 \\
0 & 1 & 1 & 1 & 1 & 1 & 0 & 0 \\
0 & 0 & 1 & 1 & 1 & 1 & 0 & 1 \\
0 & 1 & 1 & 1 & 1 & 0 & 0 & 0 \\
0 & 0 & 1 & 1 & 1 & 0 & 0 & 1 \\
1 & 2 & 3 & 3 & 2 & 2 & 1 & 1 \\
1 & 2 & 3 & 3 & 2 & 1 & 1 & 1 \\
1 & 2 & 3 & 3 & 2 & 1 & 0 & 1 \\
1 & 1 & 1 & 1 & 0 & 0 & 0 & 0 \\
1 & 2 & 3 & 2 & 2 & 2 & 1 & 2 \\
1 & 2 & 3 & 2 & 2 & 1 & 1 & 2 \\
1 & 2 & 3 & 2 & 2 & 1 & 0 & 2 \\
1 & 1 & 1 & 0 & 0 & 0 & 0 & 1 \\
1 & 2 & 3 & 3 & 3 & 2 & 1 & 1 \\
1 & 1 & 1 & 1 & 1 & 1 & 1 & 0 \\
1 & 1 & 1 & 1 & 1 & 1 & 0 & 0 \\
1 & 1 & 1 & 1 & 1 & 0 & 0 & 0 \\
1 & 3 & 5 & 4 & 3 & 2 & 1 & 3 \\
1 & 2 & 3 & 2 & 1 & 1 & 1 & 2 \\
1 & 2 & 3 & 2 & 1 & 1 & 0 & 2 \\
1 & 2 & 3 & 2 & 1 & 0 & 0 & 2 \\
0 & 1 & 1 & 0 & 0 & 0 & 0 & 0 \\
0 & 0 & 1 & 1 & 0 & 0 & 0 & 0 \\
0 & 1 & 2 & 1 & 0 & 0 & 0 & 1 \\
0 & 1 & 0 & 0 & 0 & 0 & 0 & 0 \\
0 & 0 & 0 & 1 & 0 & 0 & 0 & 0 \\
0 & 1 & 1 & 1 & 0 & 0 & 0 & 1 \\
2 & 4 & 6 & 4 & 3 & 2 & 1 & 3 \\
2 & 4 & 5 & 4 & 3 & 2 & 1 & 3 \\
2 & 3 & 5 & 4 & 3 & 2 & 1 & 2 \\
2 & 3 & 4 & 4 & 3 & 2 & 1 & 2 \\
0 & 1 & 2 & 1 & 1 & 1 & 1 & 1 \\
0 & 1 & 1 & 1 & 1 & 1 & 1 & 1 \\
0 & 0 & 1 & 1 & 1 & 1 & 1 & 0 \\
0 & 0 & 0 & 1 & 1 & 1 & 1 & 0 \\
0 & 1 & 2 & 1 & 1 & 1 & 0 & 1 \\
0 & 1 & 1 & 1 & 1 & 1 & 0 & 1 \\
0 & 0 & 1 & 1 & 1 & 1 & 0 & 0
\end{tabular}
\end{subtable}
\begin{subtable}{0.3\textwidth}
\centering
\begin{tabular}{llllllll}
0 & 0 & 0 & 1 & 1 & 1 & 0 & 0 \\
0 & 1 & 2 & 1 & 1 & 0 & 0 & 1 \\
0 & 1 & 1 & 1 & 1 & 0 & 0 & 1 \\
0 & 0 & 1 & 1 & 1 & 0 & 0 & 0 \\
0 & 0 & 0 & 1 & 1 & 0 & 0 & 0 \\
1 & 2 & 4 & 3 & 2 & 2 & 1 & 2 \\
1 & 2 & 3 & 3 & 2 & 2 & 1 & 2 \\
1 & 2 & 3 & 2 & 2 & 2 & 1 & 1 \\
1 & 2 & 2 & 2 & 2 & 2 & 1 & 1 \\
1 & 2 & 4 & 3 & 2 & 1 & 1 & 2 \\
1 & 2 & 3 & 3 & 2 & 1 & 1 & 2 \\
1 & 2 & 3 & 2 & 2 & 1 & 1 & 1 \\
1 & 2 & 2 & 2 & 2 & 1 & 1 & 1 \\
1 & 2 & 4 & 3 & 2 & 1 & 0 & 2 \\
1 & 2 & 3 & 3 & 2 & 1 & 0 & 2 \\
1 & 2 & 3 & 2 & 2 & 1 & 0 & 1 \\
1 & 2 & 2 & 2 & 2 & 1 & 0 & 1 \\
1 & 1 & 2 & 1 & 0 & 0 & 0 & 1 \\
1 & 1 & 1 & 1 & 0 & 0 & 0 & 1 \\
1 & 1 & 1 & 0 & 0 & 0 & 0 & 0 \\
1 & 1 & 0 & 0 & 0 & 0 & 0 & 0 \\
1 & 1 & 2 & 1 & 1 & 1 & 1 & 1 \\
1 & 1 & 1 & 1 & 1 & 1 & 1 & 1 \\
1 & 2 & 3 & 2 & 1 & 1 & 1 & 1 \\
1 & 2 & 2 & 2 & 1 & 1 & 1 & 1 \\
1 & 1 & 2 & 1 & 1 & 1 & 0 & 1 \\
1 & 1 & 1 & 1 & 1 & 1 & 0 & 1 \\
1 & 2 & 3 & 2 & 1 & 1 & 0 & 1 \\
1 & 2 & 2 & 2 & 1 & 1 & 0 & 1 \\
1 & 2 & 4 & 3 & 3 & 2 & 1 & 2 \\
1 & 2 & 3 & 3 & 3 & 2 & 1 & 2 \\
1 & 3 & 5 & 4 & 3 & 2 & 1 & 2 \\
1 & 3 & 4 & 4 & 3 & 2 & 1 & 2 \\
1 & 1 & 2 & 1 & 1 & 0 & 0 & 1 \\
1 & 1 & 1 & 1 & 1 & 0 & 0 & 1 \\
1 & 2 & 3 & 2 & 1 & 0 & 0 & 1 \\
1 & 2 & 2 & 2 & 1 & 0 & 0 & 1 \\
0 & 0 & 0 & 0 & 0 & 0 & 0 & 1 \\
0 & 0 & 1 & 0 & 0 & 0 & 0 & 0 \\
0 & 0 & 1 & 0 & 0 & 0 & 0 & 1
\end{tabular}
\end{subtable}
\end{table}

\begin{figure}[h]
\centering
\begin{tikzpicture}
\begin{scope}[xscale=0.25, yscale=0.25]
\fill(-8,1)circle (5pt);
\node[above] at (-8,1) {$2$};
\node[below] at (-8,1) {$\alpha_{1}$};
\fill(-4,1)circle (5pt);
\node[above] at (-4,1) {$4$};
\node[below] at (-4,1) {$\alpha_{2}$};
\fill(0,1)circle (5pt);
\node[above] at (0,1) {$6$};
\node[below] at (-1,1) {$\alpha_{3}$};
\fill(0,-3)circle (5pt);
\node[left] at (0,-3) {$3$};
\node[right] at (0,-3) {$\alpha_{8}$};
\fill(4,1)circle (5pt);
\node[above] at (4,1) {$5$};
\node[below] at (4,1) {$\alpha_{4}$};
\fill(8,1)circle (5pt);
\node[above] at (8,1) {$4$};
\node[below] at (8,1) {$\alpha_{5}$};
\fill(12,1)circle (5pt);
\node[above] at (12,1) {$3$};
\node[below] at (12,1) {$\alpha_{6}$};
\fill(16,1)circle (5pt);
\node[above] at (16,1) {$2$};
\node[below] at (16,1) {$\alpha_{7}$};
\fill(20,1)circle (5pt);
\node[below] at (20,1) {$-\widetilde{\alpha}$};

\draw[thick](-8,1)to(-4,1);
\draw[thick](-4,1)to(0,1);
\draw[thick](0,1)to(4,1);
\draw[thick](0,1)to(0,-3);
\draw[thick](4,1)to(8,1);
\draw[thick](8,1)to(12,1);
\draw[thick](12,1)to(16,1);
\draw[thick](16,1)to(20,1);

\end{scope}
\end{tikzpicture}
\caption{Extended Dynkin diagram of $E_{8}$ root system, where the dominant root $\widetilde{\alpha}=\sum_{i=1}^{8}m_{i}\alpha_{i}=(24654323)$ and the integers represent $m_{i}$'s in Theorem~\ref{thm:BdS}}\label{fig:e8dd}
\end{figure}

\subsubsection{Case 1: $(G,H)=(E_{8},D_8)$}
\label{sect:9.3.1}
In this case, according to Borel-deSiebenthal theory,
the positive root system embedding of $D_{8}$ (of cardinality $56$) to $E_{8}$ corresponds to the vertex $\alpha_{1}$ of the extended Dynkin diagram (Figure~\ref{fig:e8dd}) with $m_{1}=2$.
This is given by the set of roots such that $n_{1}$ is even in Table~\ref{fig:e8}. 
Therefore, $\Delta_{G,H}^{+}$ consists of the roots such that $n_{1}=1$ in Table~\ref{fig:e8}.
Then, $\Delta^{+}_{G,H}$ has the following linearly dependent quadruple:
\[
(12332211)-(12332111)=(12322211)-(12322111).
\]
Hence, $k(G/H)=2,3$.
By an argument similar to that used in the case $(G,H)=(E_{6},A_{5}\times A_{1})$, we also obtain $k(G/H)\not=2$.
Therefore, $k(G/H)=3$ and $(E_{8},D_{8})$ is one of the cases in Theorem~\ref{main} (2).

\subsubsection{Case 2: $(G,H)=(E_{8},E_{7}\times A_{1})$}
\label{sect:9.3.2}
In this case, according to Borel-deSiebenthal theory,
the positive root system embedding of $E_{7}\times A_{1}$ (of cardinality $64$) to $E_{8}$ corresponds to the vertex $\alpha_{7}$ of the extended Dynkin diagram (Figure~\ref{fig:e8dd}) with $m_{7}=2$.
This is given by the set of roots such that $n_{7}$ is even in Table~\ref{fig:e8}. 
Therefore, $\Delta_{G,H}^{+}$ consists of the roots such that $n_{7}=1$  in Table~\ref{fig:e8}.
One has the linearly dependent quadruple in $\Delta^{+}_{G,H}$ from the previous case.
By an argument similar to that used in the case $(G,H)=(E_{6},A_{5}\times A_{1})$, we also obtain $k(G/H)\not=2$.
Therefore, $k(G/H)=3$ and $(E_{8},E_{7}\times A_{1})$ is one of the cases in Theorem~\ref{main} (2).

\smallskip
\noindent {\bf Statements and declarations:} The authors have no competing interests to declare that are relevant to the content of this article.
Data sharing is not applicable to this article as no datasets were generated or analyzed during the current study. 


\end{document}